\documentclass[12pt]{article}

\usepackage{amsfonts,amsmath,amsthm,amssymb}
\usepackage{epsfig}
\usepackage{enumerate}

\newtheorem{theorem}{Theorem}

\newtheorem{corollary}{Corollary}
\newtheorem{lemma}{Lemma}
\newtheorem{proposition}{Proposition}

\theoremstyle{definition}

\newtheorem{example}{Example}
\newtheorem{note}{Note}

\newcommand{\R}{\mathbb{R}}

\newcommand{\nbar}{\overline{n}}
\newcommand{\rfg}{\mathcal{R}_n}
\newcommand{\fA}{\mathfrak{A}}
\newcommand{\fB}{\mathfrak{B}}
\newcommand{\spry}{^{\prime}}
\newcommand{\mbar}{\overline{m}}

\begin{document}

\title{Some properties of multivariate measures of concordance}         
 
\author{M. D. Taylor}        

\date{\today}          
 
\maketitle

\begin{abstract}
 We explore the consequences of a set of axioms which extend Scarsini's axioms for bivariate measures of concordance to the multivariate case and exhibit the following results:  (1) A method of extending measures of concordance from the bivariate case to arbitrarily high dimensions.  (2) A formula expressing the measure of concordance of the random vectors $(\pm X_1,\cdots,\pm X_n)$ in terms of the measures of concordance of the ``marginal'' random vectors $(X_{i_1},\cdots,X_{i_k})$.  (3) A method of expressing the measure of concordance of an odd-dimensional copula in terms of the measures of concordance of its even-dimensional marginals.  (4) A family of relations which exist between the measures of concordance of the marginals of a given copula.
\end{abstract}



\allowdisplaybreaks

\section{Introduction}
Scarsini proposed axioms for a bivariate measure of concordance in \cite{Scarsini84}.  The question of extending such axioms to a multivariate setting is a natural one and was explicitly raised in \cite{Nelsen02}. Adding substance to this question, plausible generalizations of well-known bivariate measures of concordance, such as Spearman's rho and Kendall's tau, were given in \cite{Dolati-Ubeda04}, \cite{Dolati-Flores04}, \cite{Joe90}, \cite{Nelsen02}, and \cite{Ubeda04}.    

As possible answers to the question, axioms for multivariate measures of concordance were proposed in \cite{Dolati-Flores04} and \cite{taylor04a}.  The axiom set of \cite{Dolati-Flores04} is essentially a subset of \cite{taylor04a}; the second set contains two extra axioms and is far more restrictive.  

This present work should be viewed as a continuation of \cite{taylor04a}.  The purpose of that paper was to present its axiom set, to give justifications for the axioms, and to show how particular functions which had been proposed in the literature as multivariate measures of concordance either satisfied or failed to satisfy the axioms.  In the present work, by developing some of their consequences, we seek to show that the axiom set of \cite{taylor04a} leads to a theory which is both rich and interesting.

What is a measure of concordance?  We give a rough and incomplete answer.

Let $(X_1,\cdots,X_n)$ be a random vector where each $X_i$ is a continuous random variable.  Intuitively speaking, we want a measure of concordance to be a function $\kappa$ operating on such continuous random vectors to produce a real number $\kappa(X_1,\cdots,X_n)$ that indicates the tendency of all the $X_i$'s to be simultaneously ``large'' or simultaneously ``small.''  We want the largest value of $\kappa(X_1,\cdots,X_n)$ to be 1 and to occur when each $X_i$ is a.s. a monotone increasing function of every other $X_j$; that is the state of maximum concordance.  If the $X_i$'s are independent, then we want to have $\kappa(X_1,\cdots,X_n)=0$; this is total lack of concordance.  If $\kappa(X_1,\cdots,X_n)$ is negative, we shall say that $(X_1,\cdots,X_n)$ is discordant.

To be more specific, the axioms of \cite{Dolati-Flores04}, \cite{Scarsini84}, and \cite{taylor04a} lead to measures of concordance that operate on the $n$-copula $C$ of $(X_1,\cdots,X_n)$ rather than on the random vector itself.  We assume familiarity with the concept of a copula but touch very briefly on some of its most salient characteristics:

The copula $C$ of $(X_1,\cdots,X_n)$ is the function $C:I^n \rightarrow I$, where $I = [0,1]$, which satisfies 
\[
  C(F_1(x_1),\cdots,F_n(x_n)) = F(x_1,\cdots,x_n)
\]
where $F_i$ is the distribution function of $X_i$ and $F$ is the joint distribution function of $(X_1,\cdots,X_n)$.  Each $n$-copula $C$ is uniquely associated with a probability measure $\mu_C$ on $I^n$ satisfying 
\[
  \mu_C(I^p \times A \times I^q) = \lambda(A) \quad \text{and} \quad 
    \mu_C([0,t_1] \times\cdots\times [0,t_n]) = C(t_1,\cdots,t_n),
\]
where $p+q=n-1$, A is a Borel set of $I$, and $\lambda(A)$ is the 1-dimensional Lebesgue measure of $A$.  If Cop($n$) is the set of $n$-copulas, then two standard elements of Cop($n$) are $M$ and $\Pi$ which are defined by 
\[
  M(t_1,\cdots,t_n) = \min(t_1,\cdots,t_n) \quad \text{and} \quad 
    \Pi(t_1,\cdots,t_n) = t_1 \cdots t_n.
\]
In line with our earlier comments, we shall require our measures of concordance to satisfy 
\[
  \kappa(M) = 1 \quad \text{and} \quad \kappa(\Pi) = 0.
\]
Standard references for copulas are \cite{Nelsenbook} and \cite{schweizsklar83}.

Now in this present work, we devote a considerable part of our time---the next two sections---to developing special notations.  This is largely occasioned by the fact that we deal so much with symmetries of the $n$-dimensional cube $I^n$ acting on copulas and with marginals of copulas, that we therefore need a good means of keeping track of indices under various transformations of copulas.  We also need a special notation for certain expressions involving measures of concordance of marginals of copulas because these expressions and variants of them occur again and again.

Once we have our special notations, in the next-to-last section, we derive interesting properties of copulas:
\begin{itemize}
  \item  Every bivariate measure of concordance in Scarsini's sense can be extended to a multivariate measure of concordance satisfying our axioms.

  \item  A formula is given which takes a copula $C$ subjected to a ``reflection'' of $I^n$ to produce a new copula $\xi^*(C)$ and computes the measure of concordance of the new copula in terms of the measures of concordance of the marginals of $C$.  To restate this idea in terms of random variables, suppose we are given a random vector $(\epsilon_1 X_1,\cdots,\epsilon_n X_n)$ where each $X_i$ is a continuous random variable and each $\epsilon_i = \pm 1$.  Then a formula is developed that expresses the measure of concordance of $(\epsilon_1 X_1,\cdots,\epsilon_n X_n)$ in terms of the measures of concordance of the random vectors $(X_{i_1},\cdots,X_{i_k})$ where $1 \leq k \leq n$ and $i_1 < \cdots < i_k$.

  \item  A numerical relation is shown to exist between measures of concordance of ``complementary'' marginals of a copula $C$.

  \item  The measure of concordance of any odd-dimensional copula $C$ is expressed as a function of the measures of concordance of the even-dimensional marginals of $C$.  

  \item  An ``asymptotic'' result is given, the measure of concordance of $(X_1,\cdots,X_n)$ being calculated as $n \rightarrow \infty$ when $X_i = -X_1$ for a fixed number $s$ of the $X_i$'s and all the other $X_i = X_1$.
\end{itemize}

It can be seen from the development given here that we are left with a number of interesting questions about measures of concordance.  We avail ourselves of the last section to raise a few of those questions.

\section{The space $V(I^n)$ and symmetries of $I^n$}

\subsection{Inequalities and rectangles}
It will be convenient for us to extend the notation for inequalities and intervals to an $n$-dimensional setting.  If $x=(x_1,\cdots,x_n)$ and $y=(y_1,\cdots,y_n)$ are points in $I^n$ (or $\R^n$), then by $x<y$ we mean $x_i<y_i$ for $i=1,\cdots,n$.  Likewise, $x\leq y$ means $x_i\leq y_i$ for $i=1,\cdots,n$, and we attach analogous meanings to $x>y$ and $x \geq y$.  Later, when dealing with random vectors $X=(X_1,\cdots,X_n)$ and $Y=(Y_1,\cdots,Y_n)$, we will feel free to use notations such as $X<Y$, $X \geq Y$, etc. in the sense indicated here.  If $x \leq y$, we shall also make use of extended interval notations to describe $n$-dimensional rectangles.  For example, $[x,y]$ will denote the $n$-dimensional (possibly degenerate) rectangle $[x_1,y_1]\times\cdots\times[x_n,y_n]$, and $[x,y)$ will denote $[x_1,y_1)\times\cdots\times[x_n,y_n)$.  We shall use $[0,x]$ and $[x,1]$ (and other fairly obvious variations on this notation) to mean $[0,x_1]\times\cdots\times[0,x_n]$ and $[x_1,1]\times\cdots\times[x_n,1]$, etc.

\subsection{Symmetries of the unit $n$-cube}
Let $I=[0,1]$.  By $I^n$ we mean, of course, the $n$-fold cartesian product of the unit interval with itself, $I{\times}\cdots{\times}I$.  By a \emph{symmetry of $I^n$} we understand a one-to-one, onto map $\xi:I^n \rightarrow I^n$ of the form 
\[
  \xi(x_1,\cdots,x_n) = (u_1,\cdots,u_n) 
\]
where for each $i$ 
\[
  u_i = 
  \begin{cases}
    x_{k_i} \text{ or} \\
    1-x_{k_i}
  \end{cases}
\]
where $(k_1,\cdots,k_n)$ is a permutation of $(1,\cdots,n)$.
By $\mathcal{S}(I^n)$ we mean the group of such symmetries under the 
operation of composition.

We say that $\xi$ is a \emph{permutation} if for each $i$ we have 
$u_i = x_{k_i}$ and is a \emph{reflection} if for each $i$ we have 
$u_i = x_i$ or $1-x_i$.  The sets of permutations and reflections constitute 
subgroups of $\mathcal{S}(I^n)$ that we label $\mathcal{P}_n$ and 
$\rfg$ respectively.  

If $\tau:I^n \rightarrow I^n$ is the permutation 
$\tau(x_1,\cdots,x_n) = (x_{k_1},\cdots,x_{k_n})$, then it is uniquely 
associated with the permutation 
\[
  (1,\cdots,n) \mapsto (k_1,\cdots,k_n)
\]
of $\{1,2,\cdots,n\}$, and we use the symbol $\tau\spry$ for this second 
permutation as well.  Thus 
\[
  \tau(x_1,\cdots,x_n) = (x_{\tau\spry(1)},\cdots,x_{\tau\spry(n)}).
\]
It is straightforward to show the following: 

\begin{proposition}
  If $\tau$ and $\rho$ are permutations of $I^n$, then 
  \begin{enumerate}[\upshape (1)]
  \item $(\tau\spry)^{-1} = (\tau^{-1})\spry$.

  \item $(\tau\circ\rho)\spry = \rho\spry \circ \tau\spry$.
  \end{enumerate}
\end{proposition}

Because of this, we shall feel free to write $\tau^{\prime -1}$ for $(\tau\spry)^{-1}$ or $(\tau^{-1})\spry$.

\bigskip

We define the \emph{elementary reflections} 
$\sigma_1, \sigma_2, \cdots, \sigma_n$ by 
\[
  \sigma_i(x_1,\cdots,x_n) = (u_1,\cdots,u_n) \text{ where } u_j = 
  \begin{cases}
    1-x_j \text{ if }j=i \\
    x_j \text{ otherwise.}
  \end{cases}
\]
By $\sigma^n$ we mean the reflection $\sigma_1\sigma_2\cdots\sigma_n$; 
that is,
\[
  \sigma^n(x_1,\cdots,x_n) = (1-x_1,\cdots,1-x_n).
\]
If the choice of $n$ is clear, we shall write $\sigma$ for $\sigma^n$. 

It should be noted that $\rfg$ is an abelian group, though the same is not true for either the group of symmetries of $I^n$ or the group of its permutations.  Further, every symmetry $\xi$ of $I^n$ has a unique representation of the form $\sigma_{i_1}\cdots\sigma_{i_k}\tau$ (and another of the form $\tau^{\prime}\sigma_{j_1}\cdots\sigma_{j_k}$) where $\tau$ is a permutation and $i_1<\cdots < i_k$.  Because of this we can define 
\[
  |\xi| = \text{ the \emph{length} of }\xi = k.
\]
Informally speaking, $|\xi|$ is the minimum number of elementary reflections needed to write $\xi$.

\subsection{Extended marginals of functions}
For $n$ a natural number, define $\nbar=\{1,2,\cdots,n\}$.  Now for every $S{\subseteq}\nbar$, define a transformation $x \mapsto x_S$ of an element of $I^n$ to a new element of $I^n$ thus:  If $x=(x_1,\cdots,x_n){\in}I^n$, then 
\[
  x_S = (u_1,\cdots,u_n) \text{ where } u_i = 
  \begin{cases}
    &1 \text{ if } i \in S, \\
    &x_i \text{ if } i \notin S.
  \end{cases}
\]
For instance, if $n=5$ and $S=\{1,4\}$, then $x=(x_1,x_2,x_3,x_4,x_5)$ becomes $x_S=(1,x_2,x_3,1,x_5)$.

For every function $f:I^n \rightarrow \R$ and $S \subseteq \nbar$, we can define a new function    $f_S:I^n \rightarrow \R$ by 
\[
  f_S(x) = f(x_S).
\]
We will use the symbol $1^n$ for the constant map $f(x)=1$.  Notice that $1^{n}_S=1^n$ for all $S \subseteq \nbar$.  It is easily seen that for any $f:I^n \rightarrow \R$ there must be a maximal set $S \subseteq \nbar$ with the property that $f = f_S$.  In some cases this maximal $S$ will be $\emptyset$.  We will call this maximal $S$ the \emph{inactive set} of $f$, and when considering $f(x_1,\cdots,x_n)$ where $(x_1,\cdots,x_n) \in I^n$, we will call $x_i$ an \emph{inactive variable} of $f$ if $i \in S$.  We refer to $\nbar - S$ as the \emph{active set} of $f$ and $x_i$ as an \emph{active variable} of $f$ if $i \notin S$.

\begin{example}
If $f:I^5 \rightarrow \R$ is $f(x_1,x_2,x_3,x_4,x_5) = x_1x_2$, then the inactive set of $f$ is $\{3,4,5\}$ and the active variables of $f$ are $x_1$ and $x_2$.
\end{example}

\begin{note}
  We will occasionally find it convenient to use the notation $f \otimes g$ when $f$ and $g$ are two real-valued functions.  If $f:D_1 \rightarrow \R$ and $g:D_2 \rightarrow \R$, then $f \otimes g:D_1 \times D_2 \rightarrow \R$ is defined by $(f \otimes g)(x,y) = f(x)g(y)$ where $x \in D_1$ and $y \in D_2$.
\end{note}

The following are easily seen to be true:

\begin{proposition} \label{proposition0}
  Suppose that $f,g:I^n \rightarrow \R$, that $S,T \subseteq \nbar$ with $s = \text{card}(S)$, and    that $x \in I^n$.  Then 
  \begin{enumerate}[\upshape (1)]
    \item  $f_S(x) = f_S(x_S)$.
    \item  $(f_S)_T = f_{S \cup T}$.
    \item  \label{prop0-3}  If $S$ is the inactive set of $f$, then there exists a permutation $\tau$ of $I^n$ and a unique function $g:I^{n-s} \rightarrow \R$ such that 
\begin{enumerate}[\upshape (a)]
  \item  $\tau^{\prime -1}$ is order preserving on $\nbar - S$,
  \item  $\tau^{\prime -1}(\nbar - S) = \{1,2,\cdots,n-s\} = \overline{n-s}$,
  \item  $f=(g \otimes 1^s)\circ\tau$,
  \item  all the variables of $g$ are active.
\end{enumerate}
  \end{enumerate}
\end{proposition}
We call the unique $g$ in the last part of this proposition the \emph{proper function} of $f$ and $\tau$ a \emph{proper permutation} of $f$.

\begin{example}
  Suppose that $f:I^5 \rightarrow \R$ is a $5$-copula.  Then the active set of $f$ is $\{1,2,3,4,5\}$ and its inactive set is $\emptyset$.  Now let us take $S=\{1,4\}$ and form $f_S$.  The active set of $f_S$ is now $\{2,3,5\}$ and the inactive set is $S$.  The proper function of $f_S$ is $g(x_2,x_3,x_5) = f(1,x_2,x_3,1,x_5)$ which is a $3$-copula.  If we take $\tau:I^5 \rightarrow I^5$ to be the permutation $\tau(x_1,x_2,x_3,x_4,x_5) = (x_2,x_3,x_5,x_1,x_4)$, then we see that $f_S = (g \otimes 1^2)\circ\tau$, that $\tau^{\prime -1}(\{2,3,5\}) = \{1,2,3\}$, and that $\tau^{\prime -1}$ is increasing on the active set of $f_S$.
\end{example}

Let $V(I^n)$ be the set of continuous $f:I^n \to \R$ having the property that there exists a signed measure $\mu_f$ on $\mathcal{B}(I^n)$, the $\sigma$-algebra of Borel sets of $I^n$, such that $\mu_f([0,x]) = f(x)$.  (Recall that $[0,x]$ denotes an $n$-dimensional rectangle, possibly degenerate, in $I^n$.)  The following is easily seen to be true:

\begin{proposition} 
  \begin{enumerate}[\upshape (1)]

    \item  $V(I^n)$ is a vector space under pointwise addition and multiplication by real scalars.
    \item  $V(I^n)$ contains all the $n$-copulas. 
    \item  $V(I^n)$ is closed under the operation $f \mapsto f_S$.
  \end{enumerate}
\end{proposition}

If $f:I^n \rightarrow \R$ and $s=\text{card\,}(S)$, then we call $f_S$ an \emph{extended $(n-s)$-marginal of $f$}.  This notion is inspired by that of the marginals of a probabilistic distribution function.  However it differs from that notion in that a marginal of a probabilistic distribution usually operates on fewer variables than the original distribution function.  Here, if $f_S$ is an extended marginal of $f \in V(I^n)$, then $f$ and $f_S$ both have $I^n$ as their domain. 

For example, if $f$ is an element of $V(I^5)$ and $S=\{1,3\}$, then 
\[
  f_S(t_1,t_2,t_3,t_4,t_5) = f(1,t_2,1,t_4,t_5).
\]
For this particular marginal we also feel free to write 
\[
  f_S = f_{13},
\]
and, more generally, if $S=\{i_1,\cdots,i_s\}$ where $i_1<\cdots<i_s$, we write 
\[
  f_S = f_{i_1{\cdots}i_s}.
\]
 We will often have occasion to write $f_i$ for $f_{\{i\}}$.  

\begin{note}
The spirit of this notation is opposite to that in some other work for marginals of probability distribution functions.  For example, \cite{Joe97} and \cite{Nelsenbook}.  Where we used $f_{13}$ for a particular marginal of $f$ just above, other works would denote the same marginal $f_{245}$.  Hopefully this will spare the reader some confusion. 
\end{note}

Every $f{\in}V(I^n)$ induces a signed measure $\mu_f$ on $(I^n,\mathcal{B}(I^n))$, where $\mathcal{B}(I^n)$ is the $\sigma$-algebra of Borel sets of $I^n$, via the equation 
\[
  \mu_f([0,x]) = f(x).
\]
(Recall here that $[0,x]$ denotes an $n$-dimensional rectangle, possibly degenerate, in $I^n$.)

It may be enlightening to see how $\mu_f$ and $\mu_{f_S}$ compare.

\begin{example}
  Let $n=2$ and $S=\{2\}$.  For $x=(x_1,x_2)$ in $I^2$ we have 
  \[
    \mu_{f_S}([0,x]) = \mu_{f_S}([0,x_1]\times[0,x_2]) = f(x_1,1) = \mu_f([0,x_1]\times[0,1]).
  \]
  Notice that if we let $x_2{\rightarrow}0^{+}$, we obtain 
  \[
     \mu_{f_S}([0,x_1]\times\{0\}) = f(x_1,1) = \mu_f([0,x_1]\times[0,1]).
  \]
  That is, all the signed $\mu_{f_S}$-mass of $[0,x]$ is located in the horizontal line segment $[0,x_1]\times\{0\}$, and that mass results from ``squashing'' the $\mu_f$-mass of $[0,x_1]\times{I}$ onto the line segment $[0,x_1]\times\{0\}$.
\end{example}

It is easily seen that more generally we may say the following:

\begin{proposition}  \label{proposition1}
  Suppose $f \in V(I^n)$ and $S \subseteq \nbar$. Let $E_1,\cdots,E_n,F_1,\cdots,F_n$ be (possibly degenerate) intervals in $I$ of the form $[0,u]$.  Then 
\[
    \mu_{f_S}(F_1 \times \cdots \times F_n) = \mu_f(E_1 \times \cdots \times E_n)
\]
provided that 
\[
  E_i = 
  \begin{cases}
    &F_i \text{ when } i \notin S, \\
    &I \text{ when } i \in S.
  \end{cases}
\]
\end{proposition}

\subsection{The action of symmetries on functions}
For $\xi$ a symmetry of $I^n$ and $f{\in}V(I^n)$, we define a map $\xi^* : V(I^n)\rightarrow V(I^n)$ by 
\[
  [\xi^*(f)](x) = \mu_f(\xi([0,x])).
\]

In addition to considering $f_S$ where $S \subseteq \nbar$, we shall find it convenient to be able to refer to $\sigma_S^*(f)$.  What we mean by that is this:  Suppose $S = \{i_1,i_2,\cdots,i_k\}      \subseteq \nbar$ where the $i_j$'s are distinct.  Then 
\[
  \sigma_S^*(f) = (\sigma_{i_1}\sigma_{i_2}\cdots\sigma_{i_k})^*(f).
\]

\begin{proposition} \label{proposition2}
  Let $\xi$, $\eta$,and $\tau$ be symmetries of $I^n$ and $f$ and $g$ be elements of $V(I^n)$.  Let $S$ and $T$ be subsets of $\nbar$.  Then the following hold:
\begin{enumerate}[\upshape (1)]
  \item  \label{prop2-1} $(\xi\circ\eta)^*(f)=\eta^*(\xi^*(f))$; that is, $(\xi\circ\eta)^* = \eta^*\circ\xi^*$.

  \item  \label{prop2-2}  $\xi^* : V(I^n) \rightarrow V(I^n)$ is a linear transformation.

  \item \label{prop2-3}  $\tau^*(f)= f \circ\tau$ if $\tau$ is a permutation of $I^n$.

  \item \label{prop2-4}  $(\tau^*(f))_S = \tau^*(f_{\tau^{-1}(S)})$ if $\tau$ is a permutation        of $I^n$.

  \item  \label{prop2-5}  $f_i\circ\sigma_i = f_i$.

  \item  \label{prop2-6}  $(\sigma_i)^*(f) = (f_i - f)\circ\sigma_i$.

  \item  \label{prop2-7}  $((\sigma_S)^*(f))_T = ((\sigma_{S-T}^*(f))_T$.

  \item  \label{prop2-8}  $(\sigma_S)^*(f_T) = 
     \begin{cases}
       &0 \text{ if } S \cap T \ne \emptyset, \\
       & ((\sigma_S)^*(f))_T \text{ if } S \cap T = \emptyset.
     \end{cases}$

  \item  \label{prop2-9}  $\underset{\xi \in \rfg}{\sum}((\xi^*(f))\circ\xi)(x) = \mu_f(I^n)$ if $x \in (0,1]^n$.
\end{enumerate}
\end{proposition}

\begin{proof}
  \eqref{prop2-1}  We know from the definition of $\xi^*(f)$ that $\mu_{\xi^*(f)}([0,x]) =    
  \mu_f(\xi([0,x]))$.  Since sets of the form $[0,x]$ generate $\mathcal{B}(I^n)$, it follows that  $\mu_{\xi^*(f)}(E) = \mu_f(\xi(E))$ for every Borel set $E$ of $I^n$.  Then 
  \[
    [\eta^*(\xi^*(f))](x) = \mu_{\xi^*(f)}(\eta([0,x])) = \mu_f(\xi\eta([0,x])) = [(\xi\eta)^*(f)](x).
  \]

  \medskip

 \eqref{prop2-2}  We show first that $\xi$ distributes over sums.  For $x \in I^n$ we have 
  \begin{align*}
    \mu_{f+g}([0,x]) &= (f+g)(x) \\
    =& f(x) + g(x) \\
    =& \mu_f([0,x]) + \mu_g([0,x]).
  \end{align*}
  From this it follows that $\mu_{f+g} = \mu_f + \mu_g$.  Then 
  \begin{align*}
    \xi^*(f+g)(x) &= \mu_{f+g}(\xi([0,x])) \\
    =&  \mu_f(\xi([0,x])) + \mu_g(\xi([0,x])) \\
    =& \xi^*(f)(x) + \xi^*(g)(x).
  \end{align*}

Next we show that multiplication by a scalar $a$ commutes with $\xi$.  Note first that 
\[
  \mu_{af}([0,x]) = (af)(x) = a(f(x)) = a \mu_f([0.x]),
\]
so that $\mu_{af} = a \mu_f$.  Then 
\[
  \xi^*(af)(x) = \mu_{af}(\xi([0,x])) = a \mu_f(\xi([0,x])) = a \xi^*(f)(x).
\]

  \medskip

 \eqref{prop2-3}  If $\tau$ is a permutation of $I^n$ and $x \in I^n$, then 
  \[
    \tau^*(f)(x) = \mu_f(\tau([0,x])) = \mu_f([0,\tau(x)]) = f(\tau(x)).
  \]

  \medskip

  \eqref{prop2-4}  It is easily checked that for a permutation $\tau$ of $I^n$ and $x \in I^n$ one has $\tau(x_S) = (\tau(x))_{\tau^{\prime -1}(S)}$.  It then follows that 
  \begin{gather*}
    (\tau^*(f))_S(x) = \tau^*(f)(x_S) = f(\tau(x_S)) = f\left( (\tau(x))_{\tau^{\prime -1}(S)} \right) \\
    = ((f_{\tau^{\prime -1}(S)})\circ \tau)(x) = \tau^*(f_{\tau^{\prime -1}(S)})(x).
  \end{gather*}

  \medskip

  \eqref{prop2-5}  The general argument is easily seen from the case where $i=1$.  Let    $x=(x_1,\cdots,x_n) \in I^n$.  Then 
  \begin{align*}
    f_1\circ\sigma_1(x) &= f((\sigma_1(x))_{\{1\}}) \\
    =&  f((1-x_1,x_2,\cdots,x_n)_{\{1\}}) \\
    =&  f(1,x_2,\cdots,x_n) \\
    =&  f(x_{\{1\}}) \\
    =&  f_1(x).
  \end{align*}

  \medskip

  \eqref{prop2-6}  The general argument is easily seen from the case $i=1$.  Let $x =     (x_1,\cdots,x_n) \in I^n$.  Then 
  \begin{align*}
    \sigma_1^*(f)(x) &= \mu_f(\sigma_1([0,x])) \\
    =& \mu_f([1-x_1,1]\times[0,x_2]\times\cdots\times[0,x_n]) \\
    =& \mu_f(I\times[0,x_2]\times\cdots\times[0,x_n]) - 
      \mu_f([0,1-x_1]\times[0,x_2]\times\cdots\times[0,x_n]) \\
    =& f(1,x_2,\cdots,x_n) - f(1-x_1,x_2,\cdots,x_n) \\
    =& (f_1 - f)\circ\sigma_1(x)
  \end{align*}
  where we have made use of part \eqref{prop2-5} of this proposition in the last step.

  \medskip

  \eqref{prop2-7}  The general argument can be seen from the particular case $S = \{1,2\}$ and $T = \{1\}$.  For $x=(x_1,\cdots,x_n) \in I^n$ we have 
  \begin{align*}
    (\sigma_{\{1,2\}}^*(f))_1(x) &= \sigma_1^*(\sigma_2^*(f))(1,x_2,\cdots,x_n) \\
    =& \mu_{\sigma_2^*(f)}(\sigma_1([0,1]\times[0,x_2]\times\cdots\times[0,x_n])) \\
    =& \mu_{\sigma_2^*(f)}([0,1]\times[0,x_2]\times\cdots\times[0,x_n]) \\
    =& (\sigma_2^*(f))_1(x).
  \end{align*}
  Thus $(\sigma_{\{1,2\}}^*(f))_1 = (\sigma_2^*(f))_1$.

  \medskip

  \eqref{prop2-8}  For the case $S \cap T \ne \emptyset$, it suffices to consider $S = T = \{1\}$.  Then by part \eqref{prop2-6} we have 
\[
  \sigma_1^*(f_1) = ((f_1)_1 - f_1)\circ\sigma_1 = (f_1 - f_1)\circ\sigma_1 = 0.
\]

The case where $S \cap T = \emptyset$ can be seen from the particular instance $i=1$ and  $j=2$.  Let $x=(x_1,\cdots,x_n) \in I^n$.  Then   
  \begin{align*}
    (\sigma_1^*(f))_2(x) &= \sigma_1^*(f)(x_{\{2\}}) \\
    =& \mu_f(\sigma_1([0,x_1]\times I \times [0,x_3]\times\cdots\times[0,x_n])) \\
    =& \mu_f([1-x_1,1]\times I \times [0,x_3]\times\cdots\times[0,x_n]) \\
    =& \mu_{f_2}(\sigma_1([0,x])) \text{ (by Proposition \ref{proposition1})} \\
    =& \sigma_1^*(f_2)(x).
  \end{align*}

  \medskip

  \eqref{prop2-9}  Choose $x = (x_1,\cdots,x_n) \in (0,1]^n$ and set 
  \[
    E_i = \{ y \in I^n \, : \, y=(y_1,\cdots,y_n) \text{ and } y_i = x_i\} 
      \text{ and } E = \cup_{i=1}^{n}E_i.
  \]
  By the continuity of $f$ and because $x_i > 0$, we see that $\mu_f(E_i)=0$ and hence $\mu_f(E)=0$.  

  Choose $y = (y_1,\cdots,y_n) \in I^n$ such that $y_i \ne x_i$ for all $i$.  (That is, $y \notin 
  E$.)  There exists a unique reflection $\xi$ of $I^n$ such that $\xi(y)<\xi(x)$.

  (Example:  Suppose that $n=4$ so that $x=(x_1,x_2,x_3,x_4)$, and suppose that we choose           $y=(y_1,y_2,y_3,y_4)$ such that 
  \[
    y_1 < x_1, \; y_2 > x_2, \; y_3 > x_3, \; y_4 < x_4.
  \]
  This amounts to 
  \[
    y_1 < x_1, \; 1-y_2 < 1-x_2, 1-y_3 < 1-x_3, \; y_4 < x_4.
  \]
  Clearly the unique reflection $\xi$ such that $\xi(y) < \xi(x)$ is $\xi = \sigma_2\sigma_3$.)

  For this unique $\xi$ such that $\xi(y)<\xi(x)$, we have $\xi(y) \in [0,\xi(x)]$, that is,        $y \in \xi([0,\xi(x)])$.  (Recall here that $\xi^{-1}=\xi$ since $\xi$ is a reflection.)          It follows from the argument we have just given that 
  \begin{equation} \label{decomposition}
    I^n - E \subseteq \underset{\xi \in \rfg}{\cup} \xi([0,\xi(x)]).
  \end{equation}

  Now suppose $\xi$ and $\eta$ are two reflections of $I^n$ such that $\xi \ne \eta$.  There must   be some $i$ such that if $\xi(u_1,\cdots,u_n) = (v_1,\cdots,v_n)$ and $\eta(u_1,\cdots,u_n) =     (w_1,\cdots,w_n)$, then $v_i=u_i$ and $w_i=1-u_i$ (or vice versa).  Now choose $z \in I^n$ such   that 
  \[
    z = (z_1,\cdots,z_n) \in \xi([0,\xi(x)]) \cap \eta([0,\eta(x)]).
  \]
  It is easily seen that $z_i=x_i$ so that $z \in E$.  We have thus shown that for $\xi \ne \eta$,  we have 
  \[
    \xi([0,\xi(x)]) \cap \eta([0,\eta(x)]) \subseteq E.
  \]

  We now know that sets of the form $\xi([0,\xi(x)])$, as $\xi$ ranges over all the different        elements of $\rfg$, do not overlap except in $E$.  It then follows from (\ref{decomposition})     that 
  \begin{align*}
    \mu_f(I^n) =& \mu_f(I^n - E) \\
    =& \underset{\xi \in \rfg}{\sum} \mu_f(\xi([0,\xi(x)])) \\
    =& \underset{\xi \in \rfg}{\sum} \xi^*(f)(\xi(x)).
  \end{align*}
\end{proof}

\section{Measures of concordance and copulas revisited}

\subsection{Axioms for a measure of concordance}
We recall from \cite{taylor04a} that to say $A \prec_{\mathcal{C}}B$, where $A$ and $B$ are two $n$-copulas, means that $A \leq B$ and $\sigma^*(A) \leq \sigma^*(B)$ where $\sigma = \sigma^n = \sigma_1\sigma_2\cdots\sigma_n$.

We next recall from \cite{taylor04a} that by a \emph{measure of concordance} $\kappa = (\{\kappa_n\},\{r_n\})$ we mean a sequence of maps $\kappa_n:\text{Cop($n$)} \rightarrow \R$ and a sequence of numbers $\{r_n\}$, where $n \geq 2$, such that if $A,B,C,$ and $C_m$ are $n$-copulas, then the following hold:
\newcounter{axiom}
\begin{list}{\bfseries A\arabic{axiom}.}{\usecounter{axiom}}
  \item  \textbf{(Normalization)} \quad
    $\kappa_n(M)=1$ and $\kappa_n(\Pi)=0$.
  \item  \textbf{(Monotonicity)} \quad
    If $A \prec_{\mathcal{C}}B$, then $\kappa_n(A) \leq 
    \kappa_n(B)$.
  \item  \textbf{(Continuity)} \quad
    If $C_m \rightarrow C$ uniformly, then 
    $\kappa_n(C_m)\rightarrow \kappa_n(C)$ as $m\rightarrow\infty$.
  \item  \textbf{(Permutation Invariance)} \quad
    $\kappa_n(\tau^*(C)) = \kappa_n(C)$ whenever $\tau$ is a permutation.
  \item  \textbf{(Duality)} \quad
    $\kappa_n(\sigma^*(C))=\kappa_n(C)$.
  \item  \textbf{(Reflection Symmetry Property; RSP)} \;
    $\underset{\rho \in \rfg}{\sum}\kappa_n(\rho^*(C))=0$ where it should be recalled that $\rfg$ is the group of reflections of $I^n$.
  \item  \textbf{(Transition Property; TP)} \quad
    $r_n \kappa_n(C) = \kappa_{n+1}(E) + \kappa_{n+1}(\sigma_1^*(E))$ 
    whenever $E$ is an $(n+1)$-copula such that 
    $C(x_1,\cdots,x_n)= E(1,x_1,\cdots,x_n)$.
\end{list}

\begin{example}
  We give two examples from \cite{Nelsen02} of measures of concordance.  It is shown for both of them in \cite{taylor04a} that they satisfy the axioms.

First is an $n$-dimensional generalization of Spearman's rho:
\begin{equation*} \label{Nelsenrho}
  \rho_n(C) = \alpha_n \left(\int_{I^n}C \, d\Pi + \int_{I^n}\Pi \, dC 
    - \frac{1}{2^{n-1}}\right)
\end{equation*}
where the value of $\alpha_n$ is chosen to satisfy $\rho_n(M) = 1$.  From \cite{taylor04a} we have
 \[
    r_n = 2{\,}\left(\frac{n+2}{n+1}\right){\,}\left(\frac{2^n-(n+1)}{2^{n+1}-(n+2)}\right).
 \]

Second is an $n$-dimensional generalization of Kendall's tau:
\begin{equation*} \label{Nelsentau}
  \tau_n(C) = \alpha_n \left(\int_{I^n}C \, dC - \frac{1}{2^n}\right)
\end{equation*}
where again the value of $\alpha_n$ is chosen so that $\tau_n(M) = 1$.  Again referring to \cite{taylor04a}, we have 
 \[
    r_n = 2 \, \left(\frac{2^{n-1}-1}{2^n-1}\right).
 \]
\end{example}

\subsection{Subsets of $\nbar$}
We will find it useful to introduce a notation which singles out subsets of $\nbar$ of a fixed size.  
  
If $S$ is a subset of $\nbar$, then either $S=\emptyset$ or 
$S=\{i_1,\cdots,i_s\}$ where $i_1<\cdots<i_s$.  For 
$0 \leq k \leq s$ where $s=$card\,($S$), then we set 
\[
  S(k) = \{T{\,}:{\,}T{\subseteq}S \text{ and card\,}(T)=k \}.
\]
It follows that $S(0)=\{\emptyset\}$, that $S(s)=\{S\}$ when
$s=$card\,($S$), and that $S(m) = \emptyset$ when $m > \text{card\,}(S)$.

\begin{example}
	If $R = \bar{4} = \{1,2,3,4\}$, we have
\[
	R(3) = \{ \{1,2,3\}, \, \{1,2,4\}, \, \{1,3,4\}, \, \{2,3,4\}\}.
\]
\end{example}

The following counting results will be useful later.  In this lemma and later in this paper, when we use the notation $R+T$, this indicates $R \cup T$ with the understanding that $R$ and $T$ are disjoint.

\begin{lemma} \label{countinglemma}
  We take $n,s$ to be natural numbers and $p,q,r$ to be nonnegative integers.  Let $S$ be a subset of $\nbar$ such that $s = \text{card}(S)$, and suppose that $\{x_P\}$ is a set of real numbers indexed by $\{P : P \in S(p)\}$.  Then the following hold:
\begin{enumerate}[\upshape (1)]
  \item \label{counting.1}
  If $p \leq r \leq s$, then
  \[
    \underset{R \in S(r)}{\sum} \quad  \underset{P \in R(p)}{\sum} x_P = 
    \binom{s-p}{r-p} \underset{P \in S(p)}{\sum} x_P = 
    \binom{s-p}{s-r} \underset{P \in S(p)}{\sum} x_P.
  \]

  \item \label{counting.2}
  If $q+r \leq s$, then
  \[
    \underset{R \in S(r)}{\sum} \quad \underset{T \in (S-R)(q)}{\sum} x_{R+T} = 
    \binom{q+r}{r} \underset{P \in S(q+r)}{\sum} x_P = 
    \binom{q+r}{q} \underset{P \in S(q+r)}{\sum} x_P.
  \]
\end{enumerate}
\end{lemma}

\begin{proof}
  \eqref{counting.1} \quad Let $A =  \underset{R \in S(r)}{\sum} \;  \underset{P \in R(p)}{\sum} x_P$.  Every term of $A$ has the property that $P \in S(p)$.  Further, every $x_P$ such that $P \in S(p)$ must appear in $A$.  If we consider a given $P \in S(p)$, how many times can $x_P$ appear in $A$?  This amounts to asking, how many ways can we choose $R \in S(r)$ such that $P \in R(p)$, or, equivalently, how many ways can we choose $R-P \subseteq S-P$ such that $\text{card}(R-P) = r-p$.  The answer is $\binom{s-p}{r-p} = \binom{s-p}{s-r}$.

\bigskip

  \eqref{counting.2} \quad  Let $B = \underset{R \in S(r)}{\sum} \; \underset{T \in (S-R)(q)}{\sum} x_{R+T}$.  Every term of $B$ has the form $x_P$ where $P \in S(q+r)$.  Further, every $x_P$ such that $P \in S(q+r)$ must appear in $B$.  If we ask for a given $P$ how often $x_P$ appears in $B$, this amounts to asking how many ways can we choose $R$ and $T$ so that $R+T = P$ and $\text{card}(R) = r$.  The answer is $\binom{q+r}{r} = \binom{q+r}{q}$.
\end{proof}

\subsection{Proper copulas and extended marginals of copulas}
It is occasionally convenient to have $0$- and $1$-copulas, so we introduce them here.  We define the unique $1$-copula $J:I \rightarrow I$ to be the identity map, $J(t)=t$, and the unique $0$-copula to be the constant $1$.  Notice that these are natural marginals of standard copulas.  

Let $\kappa = (\{\kappa_n\},\{r_n\})$ be a measure of concordance as defined in \cite{taylor04a}.  If $C$ is an $n$-copula and    $C_S$ is an extended marginal of $C$ with $s = \text{card\,}(S)$, then we know from part \eqref{prop0-3} of Proposition \ref{proposition0} that $C_S = (A \otimes 1^s)\circ\tau$ for some permutation   $\tau$ of $I^n$ and that $A$ must be an $(n-s)$-copula which is uniquely determined provided $\tau^{\prime -1}$ is order preserving on the set $\nbar - S$.  We call $A$ the \emph{proper copula} of $C_S$.  In such circumstance we shall feel free to write  $\kappa_{n-s}(C_S)$ and shall mean $\kappa_{n-s}(A)$ by this symbol.

Notice in the definition of measure of concordance that the index $n$ runs from $2$ to $\infty$ for the two sequences $\{\kappa_n\}$ and $\{r_n\}$.  We may, if we wish, extend the index to $n=0,1$.  We simply set $\kappa_0 , \kappa_1 = 0$ and observe that all the axioms for a measure of concordance remain trivially true except possibly the first axiom, Normalization.  The reason Normalization fails is that there is only one 1-copula, $J$, and there is only a single 0-copula, $1$, and these should each be regarded as versions of $\Pi$; there is nothing in dimensions 0 and 1 that corresponds to the copula $M$.  Of course one may take the view that the Normalization axiom has the structure of an if-then statement and is vacuously satisfied in those dimensions.  As for $r_0$ and $r_1$, it does not seem to matter what values we assign them since they generally wind up multiplied by $\kappa_0$ or $\kappa_1$.

We now describe the associated proper copulas in more detail.

\begin{proposition} \label{propcop}
  Let $C$ be an $n$-copula, $S$ be a subset of $\nbar$ with $s = \text{card} (S)$, and $A$ the proper $(n-s)$-copula associated with $C_S$.  Further, let $\tau$ be the permutation of $I^n$ connecting $C_S$ and $A$ as described in part \eqref{prop0-3} of Proposition \ref{proposition0} so that $C_S = (A \otimes 1^s)\circ\tau$.  Then the following hold:
\begin{enumerate}[\upshape (1)]
  \item  \label{propcop.1}  
Let $\zeta$ be a permutation of $I^n$.  Then the proper $(n-s)$-copula associated with $\zeta^*(C_S)$ is $(\xi^{-1})^*(A)$ where $\xi\spry$ is the unique permutation of $\overline{n-s} = \{1,2,\cdots,n-s\}$ which satisfies 
\[
  \zeta\spry \circ \tau\spry \circ \xi\spry (1) < \zeta\spry \circ \tau\spry \circ \xi\spry (2) < \cdots <  \zeta\spry \circ \tau\spry \circ \xi\spry (n-s).
\]

  \item  \label{propcop.2}  
If $R$ is a subset of $\nbar$ such that $R \cap S = \emptyset$, then the proper $(n-s)$-copula associated with $\sigma_R^*(C_{S})$ is $\sigma_{\tau^{\prime -1}(R)}^*(A)$.

  \item  \label{propcop.3}
If $s+1 \leq n$ and $i \in \nbar -S$ and $B$ is the proper $(n-s-1)$-copula associated with $A_{\tau^{\prime -1}(i)}$, then $B$ is also the proper $(n-s-1)$-copula associated with $C_{S+\{i\}}$.
\end{enumerate}
\end{proposition}

\begin{proof}
  \eqref{propcop.1} \quad  We first notice that 
\[
  \zeta^*(C_S)(x_1,\cdots,x_n) = C_S(x_{\zeta\spry(1)},\cdots,x_{\zeta\spry(n)}) = 
    A(x_{\zeta\spry\circ\tau\spry(1)},\cdots,x_{\zeta\spry\circ\tau\spry(n-s)}).
\]
This tells us that the active set of $\zeta^*(C_S)$ is $\zeta\spry\circ\tau\spry(\overline{n-s})$.  Let $\xi\spry$ be the unique permutation of $\overline{n-s}$ which satisfies 
\[
  \zeta\spry\circ\tau\spry\circ\xi\spry(1) < \zeta\spry\circ\tau\spry\circ\xi\spry(2) < \cdots < \zeta\spry\circ\tau\spry\circ\xi\spry(n-s).
\]
We then extend $\xi\spry$ to the rest of $\nbar$ in any way we wish so that $\xi\spry$ becomes a permutation of $\nbar$.  We see that 
\begin{align*}
  \zeta^*(C_S)(x_1,\cdots,x_n) &= A(x_{\zeta\spry\circ\tau\spry(1)},\cdots,x_{\zeta\spry\circ\tau\spry(n-s)}) \\
  =& A\circ\xi^{-1}\circ\xi(x_{\zeta\spry\circ\tau\spry(1)},\cdots,x_{\zeta\spry\circ\tau\spry(n-s)}) \\
  =& (\xi^{-1})^*(A)(x_{\zeta\spry\circ\tau\spry\circ\xi\spry(1)},\cdots,x_{\zeta\spry\circ\tau\spry\circ\xi\spry(n-s)}) \\
  =& (((\xi^{-1})^*(A) \otimes 1^s)\circ(\zeta\circ\tau\circ\xi))(x_1,\cdots,x_n).
\end{align*}
If we write the active set of $\zeta^*(C_S)$, namely $\zeta\spry\circ\tau\spry(\overline{n-s})$, as $\{i_1,\cdots,i_{n-s}\}$ where $i_1 < \cdots < i_{n-s}$, then by the definition of $\xi\spry$ we must have $(\zeta\circ\tau\circ\xi)^{-1}(i_k) = k$.  It follows from this that $(\zeta\spry\circ\tau\spry\circ\xi\spry)^{-1}$ is order-preserving on the active set of $\zeta^*(C_S)$ and that 
\[
  (\zeta\spry\circ\tau\spry\circ\xi\spry)^{-1}(\text{active set of }\zeta^*(C_S)) = \overline{n-s}.
\]
Referring to part \eqref{prop0-3} of Proposition \ref{proposition0}, we see that this means that $(\xi^{-1})^*(A)$ is the proper $(n-s)$-copula associated with $\zeta^*(C_S)$.

\bigskip

\eqref{propcop.2} \quad  We see that 
\[
  \sigma_R^*(C_{S}) = \sigma_R^*((A \otimes 1^s)\circ\tau) 
  = (\tau \circ \sigma_R)^*(A \otimes 1^s) 
  = (\sigma_{\tau^{\prime -1}(R)} \circ \tau)^*(A \otimes 1^s).
\]
Now $R \subseteq \nbar-S$, so $\tau^{\prime -1}(R) \subseteq \tau^{\prime -1}(\nbar-S) = \overline{n-s}$.  If we take $(x_1,\cdots,x_n)$ a point in $I^n$ and set 
\[
  J_k = \begin{cases}
    &[0,x_k] \text{ if } k \notin \tau^{-1}(R), \\
    &[1-x_k,1] \text{ if } k \in \tau^{-1}(R),
    \end{cases}
\]
then we see that
\[
  \sigma_{{\tau'}^{-1}(R)}^*(A \otimes 1^s)(x_1,\cdots,x_n)  
  = \mu_A(J_1 \times \cdots \times J_{n-s}) 
  = (\sigma_{\tau^{\prime -1}(R)}^*(A) \otimes 1^s)(x_1,\cdots,x_n).
\]
Thus 
\[
  \sigma_R^*(C_{S}) 
  = \tau^*(\sigma_{\tau^{\prime -1}(R)}^*(A \otimes 1^s)) 
  = (\sigma_{\tau^{\prime -1}(R)}^*(A) \otimes 1^s) \circ \tau,
\]
so that $\sigma_{\tau^{\prime -1}(R)}^*(A)$ is seen to be the proper $(n-s)$-copula associated with $\sigma_R^*(C_{S})$.

\bigskip

\eqref{propcop.3} \quad  We note that $i$ is a member of the active set of $C_S$, namely $\nbar -S = \tau\spry(\overline{n-s})$.  It follows that $\tau^{\prime -1}(i) \in \overline{n-s}$, hence $A_{\tau^{\prime -1}(i)}$ is defined.  Set $j = \tau^{\prime -1}(i)$.  Let $B$ be the proper $(n-s-1)$-copula associated with $A_{\tau^{\prime -1}(i)}$.  We have 
\begin{gather*}
  A_{\tau^{\prime -1}(i)}(x_1, \cdots , x_{j-1},x_j,x_{j+1},\cdots ,x_n) 
  = A(x_1, \cdots , x_{j-1},1,x_{j+1},\cdots ,x_n) \\
  = B(x_1, \cdots , x_{j-1},x_{j+1},\cdots ,x_n).
\end{gather*}
The proper permutation associated with $A_{\tau^{\prime -1}(i)}$ is seen to be 
\[
  \rho\spry = 
  \left(
  \begin{matrix}
    1 & 2 & \cdots & j-1 & j & j+1 & \cdots & n-s-1 & n-s \\
    1 & 2 & \cdots & j-1 & j+1 & j+2 & \cdots & n-s & j 
  \end{matrix}
  \right).
\]
We extend this to a permutation of $\nbar$ by setting $\rho\spry(k)=k$ for $k>n-s$.  We then compute 
\begin{align*}
  (B \otimes 1^{s+1})&\circ \rho\circ\tau(x_1, \cdots , x_n) = 
    (B \otimes 1^{s+1})(x_{\tau\spry\circ\rho\spry(1)}, \cdots , x_{\tau\spry\circ\rho\spry(n)}) \\
  &= B(x_{\tau\spry\circ\rho\spry(1)}, \cdots , x_{\tau\spry\circ\rho\spry(n-s-1)}) \\
  &= B(x_{\tau\spry(1)}, \cdots ,x_{\tau\spry(j-1)},x_{\tau\spry(j+1)}, \cdots ,x_{\tau\spry(n-s)}) \\
  &= A_j(x_{\tau\spry(1)}, \cdots ,x_{\tau\spry(j-1)},x_{\tau\spry(j)},x_{\tau\spry(j+1)}, \cdots ,x_{\tau\spry(n-s)}) \\
  &= (A_j \otimes 1^s)(\tau(x)) \\
  &= (A \otimes 1^s)((\tau(x))_{\{j\}}) \\
  &= (A \otimes 1^s)(\tau(x_{\{i\}})) \\
  &= C_S(x_{\{i\}}) \\
  &= C_{S+\{i\}}(x).
\end{align*}
Since $\tau\spry\circ\rho\spry(1) < \tau\spry\circ\rho\spry(2) < \cdots < \tau\spry\circ\rho\spry(n-s-1)$, we see that $B$ is the proper $(n-s-1)$-copula associated with $C_{S+\{i\}}$.
\end{proof}

We now want to show that a measure of concordance $\kappa$ acts on a marginal of a copula, $C_T$, in ways that one might expect.  It is helpful to first note the easily verified fact that if $\tau$ is a permutation of $I^n$ and $S \subseteq \nbar$, then 
\begin{equation} \label{permref}
  \tau \circ \sigma_{\tau\spry(S)} = \sigma_S \circ \tau \quad \text{or, equivalently,} \quad 
    \tau \circ \sigma_{S} = \sigma_{\tau^{\prime -1}(S)} \circ \tau .
\end{equation}
It may also be helpful to bear in mind that, as was shown earlier,
\[
  \text{if $R$ and $S$ are disjoint, then } \sigma_S^*(C)_R = \sigma_S^*(C_R).
\] 

\begin{proposition} \label{mocmarg}
Let $C$ be an $n$-copula, and suppose $R$ and $S$ are subsets of $\nbar$, with $r = \text{card}(R)$ and $s = \text{card}(S)$, such that $R \cap S = \emptyset$ and $R \cup S = \nbar$.  (Note that $r+s=n$.)  Then the following hold:
\begin{enumerate}[\upshape (1)]
  \item \label{mocmarg.1}  If $\zeta$ is a permutation of $I^n$, then 
  \[
    \kappa_{n-r}(\zeta^*(C_R)) = \kappa_{n-r}(C_R) \quad \text{and} \quad 
    \kappa_{n-r}(\zeta^*(C)_R) = \kappa_{n-r}(C_{\zeta^{\prime -1}(R)}).
  \]
  \item \label{mocmarg.2}  $\kappa_s(\sigma_S^*(C_{R})) = \kappa_s(C_R)$.
  \item \label{mocmarg.3}  $\underset{T \subseteq S}{\sum} \kappa_s(\sigma_T^*(C_{R})) = 0$.
  \item \label{mocmarg.4}  If $i \in \nbar -R = S$ and $1 \leq s$, then
\[
  r_{s-1} \kappa_{s-1}(C_{R+\{i\}}) = \kappa_s(C_R) + \kappa_s(\sigma_i^*(C_R)).
\]
\end{enumerate}
\end{proposition}

\begin{proof}
Throughout this proof we take $A$ to be the proper $(n-r)$-copula associated with $C_R$ and $\tau$ to be a permutation of $I^n$ such as described in part \eqref{prop0-3} of Proposition \ref{proposition0} so that $C_R = (A \otimes 1^r)\circ\tau$. 

	\eqref{mocmarg.1} \quad  
We know from Proposition \ref{propcop} that the proper $(n-r)$-copula associated with $\zeta^*(C_R)$ has the form $\psi^*(A)$ where $\psi$ is a permutation of $I^{n-r}$.  Then making use of the Permutation Invariance axiom for measures of concordance, we have 
\[
  \kappa_{n-r}(\zeta^*(C_R)) = \kappa_{n-r}(\psi^*(A)) = \kappa_{n-r}(A) 
  = \kappa_{n-r}(C_R).
\]

Next we know that $\zeta^*(C)_R = \zeta^*(C_{\zeta^{\prime -1}(R)})$.  Thus, appealing to the first half of this proof, we have  
\[
   \kappa_{n-r}((\zeta^*(C)_R) = \kappa_{n-r}(\zeta^*(C_{\zeta^{\prime -1}(R)})) 
    = \kappa_{n-r}(C_{\zeta^{\prime -1}(R)}),
\]
and we are done.

\bigskip

	\eqref{mocmarg.2} \quad
By Proposition \ref{propcop} the proper $(n-r)$-copula associated with $\sigma_S^*(C_{R})$ is $\sigma_{\tau^{\prime -1}(S)}^*(A)$.  Now we know that $S = \nbar - R$ and $\tau^{\prime -1}(\nbar - R) = \overline{n-r} = \overline{s}$, so that $\sigma_{\tau^{-1}(S)}^*(A) = \sigma_{\overline{s}}^*(A)$.  Since $A$ is an $s$-copula, this means we can use the Duality axiom for measures of concordance in the calculation 
\[
  \kappa_s(\sigma_S^*(C_R)) = \kappa_s(\sigma_{\overline{s}}^*(A)) = \kappa_s(A) = \kappa_s(C_R),
\]
and we are done.

\bigskip

	\eqref{mocmarg.3} \quad
Following the notation of the last part, we have 
\[
  \underset{T \subseteq S}{\sum} \kappa_s(\sigma_T^*(C_R)) 
  = \underset{T \subseteq S}{\sum} \kappa_s(\sigma_{\tau^{\prime -1}(T)}^*(A)) 
  = \underset{P \subseteq \overline{s}}{\sum} \kappa_s(\sigma_P^*(A)).
\]
But as $P$ runs through all subsets of $\overline{s}$, we see that $\sigma_P$ runs through all the reflections of $I^s$, so that by the Reflection Symmetry Property of measures of concordance we have 
\[
  \underset{P \subseteq \overline{s}}{\sum} \kappa_s(\sigma_P^*(A)) = 0.
\]
Thus we are done.

\bigskip

	\eqref{mocmarg.4} \quad 
Let $B$ be the proper $(s-1)$-copula for $A_{\tau^{\prime -1}(i)}$.  We know from part \eqref{propcop.3} of Proposition \ref{propcop} that $B$ is also the proper $(s-1)$-copula for $C_{R+\{i\}}$.  By the Transition Property for copulas we have 
\[
  r_{s-1} \kappa_{s-1}(B) = \kappa_s(A) + \kappa_s(\sigma_{\tau^{\prime -1}(i)}^*(A)).
\]
By the uniqueness of proper copulas and part \eqref{propcop.2} of Proposition \ref{propcop}, we know that $\sigma_{\tau^{\prime -1}(i)}^*(A)$ is the proper $s$-copula associated with $\sigma_i^*(C_R)$.  Thus the last equation becomes 
\[
  r_{s-1} \kappa_{s-1}(C_{R+\{i\}}) = \kappa_s(C_R) + \kappa_s(\sigma_i^*(C_R)),
\]
and we are done.
\end{proof}

We now touch on the relationship between stochastic inequality and extended marginals of copulas.

We recall from \cite{taylor04a} that to say $A \prec_{\mathcal{C}}B$, where $A$ and $B$ are two $n$-copulas, means that $A \leq B$ and $\sigma^*(A) \leq \sigma^*(B)$ where $\sigma = \sigma_{\nbar} = \sigma_1\sigma_2\cdots\sigma_n$.  If $S \subseteq \nbar$, then we define $A_S  \prec_{\mathcal{C}} B_S$ to mean that $A\spry  \prec_{\mathcal{C}} B\spry$ where $A\spry$ and $B\spry$ are the proper copulas associated with $A$ and $B$ respectively.

\begin{proposition} \label{margmono}
  If $A$ and $B$ are $n$-copulas such that $A  \prec_{\mathcal{C}} B$, then for all $S \subseteq \nbar$ we have $A_S  \prec_{\mathcal{C}} B_S$.
\end{proposition}

\begin{proof}
  Choose $S \subseteq \nbar$.  Let $s = \text{card}(S)$ and $m = n-s$.  We let $A\spry$ and $B\spry$ denote the proper $m$-copulas of $A_S$ and $B_S$ respectively.  If $\tau$ is a proper permutation of $A_S$, then we can write $A_S = (A\spry \otimes I^s)\circ\tau$.  Since $A_S$ and $B_S$ have the same inactive set, they have the same active set; hence $\tau$ must be a proper permutation for $B_S$ as well as $A_S$, and we can write $B_S = (B\spry \otimes I^s)\circ\tau$.

It then follows from $A \leq B$ that $A\spry \otimes I^s \leq B\spry \otimes I^s$, and hence $A\spry \leq B\spry$.

We now need only show that $\sigma_{\mbar}^*(A\spry) \leq \sigma_{\mbar}^*(B\spry)$.  We know that $\sigma_{\nbar}^*(A) \leq \sigma_{\nbar}^*(B)$, so for any $x \in I^n$ we have 
\[
  \sigma_{\nbar}^*(A)_S(x) = \sigma_{\nbar}^*(A)(x_S) \leq 
    \sigma_{\nbar}^*(B)(x_S) = \sigma_{\nbar}^*(B)_S(x).
\]
By imitating the proof of Proposition \ref{propcop}\eqref{propcop.2}, we can show that 
\[
  \sigma_{\nbar}^*(A)_S = (\sigma_{\mbar}^*(A\spry) \otimes 1^s)\circ\tau \quad \text{and} 
  \quad \sigma_{\nbar}^*(B)_S = (\sigma_{\mbar}^*(B\spry) \otimes 1^s)\circ\tau.
\]
It then follows from $\sigma_{\nbar}^*(A)_S \leq \sigma_{\nbar}^*(B)_S$ that $\sigma_{\mbar}^*(A\spry) \leq \sigma_{\mbar}^*(B\spry)$.
\end{proof}

\subsection{The $\fA^{S}_{j}(C_T)$ notation}

The quantities that we now describe arise later, in a natural way, when we show how to compute the measure of concordance of $\xi^*(C)$, where $\xi$ is a reflection, in terms of the measures of concordance of marginals of $C$, or in deriving relations such as the \'{U}beda identities.  

Let $\kappa = (\{\kappa_n\},\{r_n\})$ be a given measure of concordance.  Now let us suppose that 
\begin{gather*}
	C \text{ is an $n$-copula,} \\
	S, T \text{ are subsets of } \nbar, \\
	\text{ and } \; t = \text{card\,}(T).
\end{gather*}
In everything that follows, we permit ourselves to indicate disjoint unions with the symbol $+$ instead of $\cup$.  That is, $S+T$ denotes $S \cup T$ with the understanding that $S \cap T = \emptyset$.

We define 
\begin{equation} \label{frakA}
	\fA^{S}_{j}(C_T) = r_{n-t} r_{n-t-1} \cdots r_{j} 
	  \underset{R \in (S-T)(n-j-t)}{\sum} \kappa_{j}(C_{R+T})
\end{equation}
where this expression makes sense if $0 \leq j,t \leq n$.  If $(S-T)(n-j-t) = \emptyset$, then we take $\fA^{S}_{j}(C_T)$ to be $0$.  If $S = \nbar$ or $T = \emptyset$, then we feel free to omit them in this notation.  Thus 
\begin{gather*}
	\fA^{S}_j(C) = r_n r_{n-1} \cdots r_{j}\underset{R \in S(n-j)}{\sum} \kappa_{j}(C_{R}), \\
	\fA_{j}(C_T) = r_{n-t} r_{n-t-1} \cdots r_{j}\underset{R \in (\nbar -T)(n-j-t)}{\sum} \kappa_{j}(C_{R+T}), \\
	\fA_j(C) = r_n r_{n-1} \cdots r_{j}\underset{T \in \nbar(n-j)}{\sum} \kappa_{j}(C_{T}).
\end{gather*}
For example suppose $C$ is a 4-copula, that $S = \{1,2,3\}$, and $T = \{1\}$.  Then 
\begin{gather*}
	\fA_2(C) = r_4 r_3 r_2 [ \kappa_2(C_{12}) + \kappa_2(C_{13}) + \kappa_2(C_{14}) + 
	  \kappa_2(C_{23}) + \kappa_2(C_{24}) + \kappa_2(C_{34})], \\
	\fA_{2}^S(C) =  r_4 r_3 r_2 [ \kappa_2(C_{12}) + \kappa_2(C_{13}) + \kappa_2(C_{23})], \\
	\fA_{2}^S(C_T) = r_3  r_2 [ \kappa_2(C_{12}) + \kappa_2(C_{13})].
\end{gather*}

\bigskip

\begin{note}
By the definition of $\fA_j^S(C_T)$ we have $\fA_j^S(C_T) = \fA_j^{S-T}(C_T)$, therefore we may, if we wish, when using this notation, assume that $S$ and $T$ are disjoint.
\end{note}

\bigskip

If $C_T$ is an extended copula and $A$ is the associated proper copula, then $\fA_{j}^S$ evaluated at $C_T$ should be thought of as operating on $A$.  For example, suppose $n=m+2$, that $A$ is an $m$-copula, and $C_T = A \otimes 1^2$.  Then 
\[
	\fA_j(C_T) = \fA_j(A) = r_m \cdots r_j 
	  \underset{R \in \overline{m}(m-j)}{\sum} \kappa_{j}(A_{R}).
\]

We first give some of the basic properties of the $\fA^{S}_{j}(C_T)$ notation.

\begin{proposition} \label{basicfA}
	Let $C$ be an $n$-copula and $R, S, T$ be subsets of $\nbar$ such that $r = \text{card\,}(R)$, $s = \text{card\,}(S)$, and $t = \text{card\,}(T)$.  Then the follwing hold: 
	\begin{enumerate}[\upshape (1)]
	  \item \label{basicfA.1}  $\fA_{n-t}(C_T) = \fA_{n-t}^S(C_T) = r_{n-t} \kappa_{n-t}(C_T).$ (In particular, this is true for $S = \emptyset$.) It follows that 
\[
  \fA_j^S(C_T) = r_{n-t}r_{n-t+1}\cdots r_{j+1} \underset{P \in S(n-t-j)}{\sum}\fA_j^Q(C_{T+P})
\]
where $0 \leq j \leq n-t$, it is understood that $r_{n-t}r_{n-t+1}\cdots r_{j+1}=1$ if $j=n-t$, and $Q$ is any subset of $\nbar$.

	  \item \label{basicfA.2}  $\fA_{n-r}^S(C_T) = 0$ if $r>s+t$.  

	  \item \label{basicfA.4}  If $i \notin S,T$, then 
	  \[
	    \fA_{j}^{\{i\}+S}(C_T) = \fA_{j}^S(C_T) + r_{n-t} \fA_{j}^S(C_{\{i\}+T})
          \]
	  where $0 \leq j \leq n-t-1$.

	  \item \label{basicfA.5}  If $S \cap T = \emptyset$ and $0 \leq n-j-t \leq r \leq s \leq n$, then 
          \[
            \underset{R \in S(r)}{\sum} \fA_{j}^R(C_T) = \binom{s+j+t-n}{s-r} \fA_{j}^S(C_T).
	  \]  

	  \item \label{basicfA.6}  If $S \cap T = \emptyset$ and $r,s,t,j$ are nonnegative integers such that $r \leq s$ and $j+t+r \leq n$, then 
	  \[
		r_{n-t} r_{n-t-1} \cdots r_{n-t-r+1} \underset{R \in S(r)}{\sum}
		\fA_{j}^S(C_{T+R}) \quad = \quad \binom{n-j-t}{r} \, \fA_j^S(C_T),
	  \]
	where it is understood that $r_{n-t} r_{n-t-1} \cdots r_{n-t-r+1} = 1$ if $r=t=0$.
	 \end{enumerate}
\end{proposition}

\begin{proof}
 \eqref{basicfA.1} \ Since $S(0) = \{\emptyset\}$ and $C_{T+\emptyset} = C_T$, we see that 
  \[
    \fA_{n-t}^S(C_T) = r_{n-t} \underset{R \in S(0)}{\sum} \kappa_{n-t}(C_{T+R}) = r_{n-t} \kappa_{n-t}(C_T).
  \]
The formula for $\fA_{j}^S(C_T)$ in terms of $\fA_j^Q(C_{T+P})$ then follows from the definition of the notation.

  \eqref{basicfA.2} \ Because $(S-T)(r-t) = \emptyset$ if $r-t>s$, we have 
  \[
    \fA_{n-r}^S(C_T) = r_{n-r} \cdots r_{n-t} 
     \underset{R \in (S-T)(r-t)}{\sum} \kappa_{n-r}(C_{R+T}) = 0.
  \]

  \eqref{basicfA.4} \ We see that 
  \begin{align*}
    \fA_{j}^{\{i\}+S}&(C_T) = r_{n-t} \cdots r_j \underset{R \in ((\{i\}+S)-T)(n-j-t)}{\sum}
	\kappa_j(C_{R+T}) \\
    =& \; r_{n-t} \cdots r_j \underset{R \in (S-T)(n-j-t)}{\sum}
	\kappa_j(C_{R+T}) + 
    r_{n-t} \cdots r_j \underset{P \in (S-T)(n-j-t-1)}{\sum}
	\kappa_j(C_{\{i\}+T+P}) \\
    =& \quad \fA_{j}^S(C_T) + r_{n-t} \fA_{j}^S(C_{\{i\}+T}).
  \end{align*}

  \eqref{basicfA.5} \ We see that 
  \begin{align*}
    \underset{R \in S(r)}{\sum}&\fA_{j}^R(C_T) = r_{n-t} \cdots r_j
      \underset{R \in S(r)}{\sum} \quad \underset{P \in R(n-j-t)}{\sum}\kappa_j(C_{T+P}) \\
      &= r_{n-t} \cdots r_j \binom{s+j+t-n}{r+j+t-n} 
      \underset{P \in S(n-j-t)}{\sum}\kappa_j(C_{T+P}) \\
    &= \binom{s+j+t-n}{s-r} \fA_{j}^S(C_T),
  \end{align*}
  where the middle step follows from Lemma \ref{countinglemma}\eqref{counting.1}.

\eqref{basicfA.6} \ We have 
\begin{align*} \label{basicfA.6-1}
  \underset{R \in S(r)}{\sum}&\fA_{j}^S(C_{T+R}) = \underset{R \in S(r)}{\sum} 
    r_{n-t-r} \cdots r_j \underset{P \in (S-R)(n-j-t-r)}{\sum} \kappa_j(C_{T+R+P}) \\
  &= r_{n-t-r} \cdots r_j \binom{n-j-t}{r} \underset{Q \in S(n-j-t)}{\sum} \kappa_j(C_{T+Q}),
\end{align*}
where the last step comes from Lemma \ref{countinglemma}\eqref{counting.2}. The conclusion follows easily.
\end{proof}

We now examine properties that $\fA_{j}^S(C_T)$ ``inherits'' from the measure of concordance $\kappa$.

\begin{proposition} \label{fAmoc}
  Let $C$ be an $n$-copula and $R,S,T$ be subsets of $\nbar$ such that $t = \text{card\,}(T)$.  Then the following hold:
  \begin{enumerate}[\upshape (1)]
  \item \label{fAmoc.1} \ If $\tau$ is a permutation of $I^n$, then 
  \[
    \fA_{j}^S((\tau^*(C))_T) = \fA_{j}^{\tau^{\prime -1}(S)}(C_{\tau^{\prime -1}(T)}) \quad \text{and} 
	\quad \fA_{j}^S(\tau^*(C_T)) = \fA_{j}^{\tau^{\prime -1}(S)}(C_T)
  \]
  where $j$ and $t$ are nonnegative integers such that $j+t \leq n$.

  \item \label{fAmoc.2} \ If $R+T = \nbar$, then $\fA_j^S(\sigma_R^*(C_T)) = \fA_j^S(C_T)$.    

  \item \label{fAmoc.3} \ $\underset{R \subseteq \nbar - T}{\sum}\fA_{j}^S(\sigma_R^*(C_T)) = 0$.

  \item \label{fAmoc.4} \ If $1+t \leq j+t \leq n$ and $i \in \nbar -(S \cup T)$, then
  \[
    r_{n-t}\fA_{j-1}^S(C_{T+\{i\}}) = \fA_j^S(C_T) + \fA_j^S(\sigma_i^*(C_T)).
  \]
  \end{enumerate}
\end{proposition}

\begin{proof}
\eqref{fAmoc.1} \ Suppose that $\tau$ is a permutation of $I^n$.  We then see that 
\begin{align*}
  \fA_{j}^S((&\tau^*(C))_T) = r_{n-t} \cdots r_j \underset{P \in (S-T)(n-j-t)}{\sum}
    \kappa_j((\tau^*(C))_{T+P}) \\
   &= \; r_{n-t} \cdots r_j \underset{P \in (S-T)(n-j-t)}{\sum}
    \kappa_j(\tau^*(C_{\tau^{\prime -1}(T)+\tau^{\prime -1}(P)})) \\
   &= \; r_{n-t} \cdots r_j \underset{P\in (S-T)(n-j-t)}{\sum}
    \kappa_j(C_{\tau^{\prime -1}(T)+\tau^{\prime -1}(P)}) \\
   &= \; r_{n-t} \cdots r_j \underset{Q \in (\tau^{\prime -1}(S)-\tau^{\prime -1}(T))(n-j-t)}{\sum}
    \kappa_j(C_{\tau^{\prime -1}(T)+Q}) \\
   &= \; \fA_{j}^{\tau^{\prime -1}(S)}(C_{\tau^{\prime -1}(T)}).
\end{align*}

The other equality follows from the fact that $\tau^*(C_T) = \tau^*(C)_{\tau\spry(T)}$.

\bigskip

\eqref{fAmoc.2} \ Using parts \eqref{prop2-8} and \eqref{prop2-7} of Proposition \ref{proposition2}, we see that 
\begin{align*}
  \fA_j^S(&\sigma_R^*(C_T)) = \fA_j^S(\sigma_R^*(C)_T) \\ 
    =& r_{n-t} \cdots r_j \underset{P \in (S-T)(n-j-t)}{\sum} \kappa_j(\sigma_R^*(C)_{T+P}) \\
    =& r_{n-t} \cdots r_j \underset{P \in (S-T)(n-j-t)}{\sum} 
	\kappa_j(\sigma_{R-P}^*(C)_{T+P}).
\end{align*}
Now $(R-P)+(T+P) = \nbar$, so by part \eqref{mocmarg.2} of Proposition \ref{mocmarg}, we have 
\[
  \fA_j^S(\sigma_R^*(C_T)) = r_{n-t} \cdots r_j \underset{P \in (S-T)(n-j-t)}{\sum}
    \kappa_j(C_{T+P}) = \fA_j^S(C_T).
\]

\bigskip

\eqref{fAmoc.3} \ Note that 
\begin{align*}
  \underset{R \subseteq \nbar - T}{\sum} &\fA_j^S(\sigma_R^*(C_T)) = 
     \underset{R \subseteq \nbar - T}{\sum} r_{n-t} \cdots r_j 
    \underset{P \in (S-T)(n-j-t)}{\sum} \kappa_j((\sigma_R^*(C))_{T+P}) \\
  =& r_{n-t} \cdots r_j \underset{P \in (S-T)(n-j-t)}{\sum} \quad
    \underset{R \subseteq \nbar - T}{\sum} \kappa_j(\sigma_{R-P}^*(C)_{T+P})
\end{align*}
where the last step is justified by part \eqref{prop2-7} of Proposition \ref{proposition2}.  Now fix $P \in (S-T)(n-j-t)$ and set $p = n-j-t = \text{card\,}(P)$.  If we are given a subset $Q$ of $\nbar - (T+P)$ and we wish to find $R$ such that $Q=R-P$ and card\,($R \cap P$)=$k$, then the number of ways we can do this is $\binom{p}{k}$.  Thus we can write 
\begin{gather*}
  \underset{R \subseteq \nbar - T}{\sum} \kappa_j(\sigma_{R-P}^*(C)_{T+P}) = 
    \sum_{k=0}^p \underset{\stackrel{R \subseteq \nbar - T}{\text{card\,}(R \cap P)=k}}{\sum}
     \kappa_j(\sigma_{R-P}^*(C)_{T+P}) \\
  = \sum_{k=0}^p \binom{p}{k} \underset{Q \subseteq \nbar - (T+P)}{\sum} 
    \kappa_j(\sigma_{Q}^*(C)_{T+P}).
\end{gather*}
But by part \eqref{mocmarg.3} of Proposition \ref{mocmarg} we have 
\[
  \underset{Q \subseteq \nbar - (T+P)}{\sum} \kappa_j(\sigma_{Q}^*(C)_{T+P}) = 0,
\]
so we are done.

\bigskip

\eqref{fAmoc.4} \ We see that 
\begin{align*}
  r_{n-t} \fA_{j-1}^S(&C_{T+\{i\}}) = 
    r_{n-t} \cdots r_{j-1} \underset{R \in (S-(T+\{i\}))(n-j-t)}{\sum} 
    \kappa_{j-1}(C_{T+\{i\}+R}) \\
  =& \quad r_{n-t} \cdots r_j \underset{R \in (S-T)(n-j-t)}{\sum} 
    [\kappa_j(C_{T+R}) + \kappa_j(\sigma_i^*(C)_{T+R})] \\
  =& \quad \fA_j^S(C_T) + \fA_j^S(\sigma_i^*(C)_T).
\end{align*}
Thus the result is established.
\end{proof}

\subsection{The $\fB^{S}_{j}(C_T)$ notation}
It is convenient to introduce a notation $\fB^{S}_{j}(C_T)$ which is a simple variation of $\fA^{S}_{j}(C_T)$:
\begin{equation} \label{frakB}
	\fB^{S}_{j}(C_T) = \underset{R \in (S-T)(n-j-t)}{\sum} \kappa_{j}(C_{R+T})
\end{equation}
where $C$ is an $n$-copula and $t = \text{card}(T)$.  Thus $\fA^{S}_{j}(C_T) = r_{n-t}\cdots r_j \fB^{S}_{j}(C_T)$.

Not surprisingly, one can prove propositions for $\fB^{S}_{j}(C_T)$ that are almost identical to Propostions \ref{basicfA} and \ref{fAmoc}.  For Proposition \ref{basicfA}, the corresponding statements are obtained by replacing $\fA$ by $\fB$ and eliminating all $r_k$'s.  Thus, for example, the equation 
\[
		r_{n-t} r_{n-t-1} \cdots r_{n-t-r+1} \underset{R \in S(r)}{\sum}
		\fA_{j}^S(C_{T+R}) \quad = \quad \binom{n-j-t}{r} \, \fA_j^S(C_T) 
\]
in Proposition \ref{basicfA}\eqref{basicfA.6} becomes 
\[
	\underset{R \in S(r)}{\sum}\fB_{j}^S(C_{T+R}) \quad = 
        \quad \binom{n-j-t}{r} \, \fB_j^S(C_T).
\]
For Proposition \ref{fAmoc}, the corresponding statements are obtained by replacing $\fA$ by $\fB$ except for part \eqref{fAmoc.4} where the equation 
\[
    r_{n-t}\fA_{j-1}^S(C_{T+\{i\}}) = \fA_j^S(C_T) + \fA_j^S(\sigma_i^*(C_T))
\]
is replaced by 
\[
    r_{j-1}\fB_{j-1}^S(C_{T+\{i\}}) = \fB_j^S(C_T) + \fB_j^S(\sigma_i^*(C_T)).
\]

One might expect these results for $\fB^{S}_{j}(C_T)$ to be proved by simply dividing $\fA^{S}_{j}(C_T)$ by $r_{n-t}\cdots r_j$, however one does not know that all the $r_k$'s are nonzero.  But a quick reading of the proofs for the different parts of Propositions \ref{basicfA} and \ref{fAmoc} shows that they can, essentially, simply be duplicated for $\fB^{S}_{j}(C_T)$.

We shall feel free to use these results and shall refer to them by phrases such as, ``the property of $\fB^{S}_{j}(C_T)$ corresponding to Proposition \ref{fAmoc} $(k)$.''

\section{Properties of a measure of concordance}

\subsection{Previous results}
We recall some results from \cite{taylor04a}.

\begin{theorem}
  For every measure of concordance $\kappa = (\{\kappa_n\},\{r_n\})$, 
  the following is true:
  \begin{enumerate}[\upshape (1)]
    \item    If $C$ and $E$ are $(n-1)$- and $n$-copulas respectively 
      (where $n{\geq}3$) such that 
      \[
        E(x_1,\cdots,x_{i-1},1,x_{i+1},\cdots,x_{n}) = 
	  C(x_1,\cdots,,x_{i-1},x_{i+1},\cdots,x_{n}),
      \]
      then $r_{n-1}\kappa_{n-1}(C) = \kappa_n(E)+\kappa_n(\sigma_i^*(E))$.
    \item    $r_{n-1} = 1+\kappa_n(\rho^*(M))$ whenever $|\rho|=1$ or 
      $n-1$ and $n{\geq}3$.
    \item   $r_2=\frac{2}{3}$ and $\kappa_3(\rho^*(M))=-\frac{1}{3}$ 
      whenever $|\rho|=1$ or $2$.
  \end{enumerate}
\end{theorem}

Some $n$-copulas $C$ have the property that for any given $k$, where $1 \leq k \leq n$, all the $k$-marginals of $C$ are equal.  Simple examples of this are the $n$-copulas $M$ and $\Pi$.  

\begin{theorem}
  If $C$ is an $n$-copula with the property that for each $k$ satisfying 
  $1{\leq}k{\leq}n$ all of its $k$-marginals are equal, then 
  for all measures of concordance $\kappa$ we have 
  $\kappa_n(\rho^*(C)) = \kappa_n(\xi^*(C))$ whenever $|\rho|=|\xi|$ or 
  $|\rho|+|\xi|=n$.
\end{theorem}

\begin{corollary} \label{Mmoc}
  For all $n{\geq}2$ and all $\rho$ and $\xi$ such that $|\rho|=|\xi|$ 
  or $|\rho|+|\xi|=n$, we have $\kappa_n(\rho^*(M))=\kappa_n(\xi^*(M))$.
\end{corollary}

\subsection{Extending a bivariate measure of concordance}

Scarsini's axioms for a bivariate measure of concordance $\kappa_2 : \text{Cop}(2) \rightarrow \R$ are equivalent to the following:
\newcounter{scarax}
\begin{list}{\bfseries S\arabic{scarax}.}{\usecounter{scarax}}
  \item  $\kappa_2(M) = 1$ and $\kappa_2(\Pi) = 0$.
  \item  If $C_m \rightarrow C$ uniformly, then 
    $\kappa_2(C_m)\rightarrow \kappa_2(C)$ as $m\rightarrow\infty$.
  \item  $\kappa_2(\tau^*(C)) = \kappa(C)$ whenever $\tau$ is a permutation of $I^2$.
  \item \label{monotone}  If $A \prec_{\mathcal{C}} B$, then $\kappa_2(A) \leq \kappa_2(B)$.
  \item  \label{signchange}  $\kappa(\sigma_1^*(C)) = \kappa(\sigma_2^*(C)) = -\kappa(C)$.
\end{list}
We will show that a bivariate measure of concordance can always be extended to a multivariate measure of concordance in our sense of the term.

\begin{theorem} \label{kextend}
Let $\kappa_2$ be a bivariate measure of concordance in the sense of Scarsini.  Define $\kappa_{2+p}:\text{Cop}(2+p) \rightarrow \R$ by 
\[
  \kappa_{2+p}(C) = \frac{1}{\binom{2+p}{2}} \, \fB_2(C) \quad \text{for } p \geq 1.
\]
Then $\kappa = (\{\kappa_n\},\{r_n\})_{n=2}^{\infty}$ is a multivariate measure of concordance with $r_{1+p} = 2p/(2+p)$.
\end{theorem}

\begin{proof}
  The Normalization axiom can be seen as follows:  If $M \in \text{Cop}(2+p)$, then 
\[
  \kappa_{2+p}(M) = \frac{1}{\binom{2+p}{2}} \underset{T \in \overline{2+p}(p)}{\sum} 
    \kappa_2(M_T) = \frac{1}{\binom{2+p}{2}} \, \binom{2+p}{2} = 1.
\]
The verification that $\kappa_{2+p}(\Pi) = 0$ is similar.

To check the Continuity axiom, let us suppose that $\{C_k\}$ is a sequence of $(2+p)$-copulas which converge uniformly to $C$.  For every $T \in \overline{2+p}(p)$ we see that $(C_k)_T \rightarrow C_T$ uniformly, and hence, since $\kappa_2$ is a bivariate measure of concordance, we have $\kappa_2((C_k)_T) \rightarrow \kappa_2(C_T)$ as $k \rightarrow \infty$.  It follows that $\kappa_{2+p}(C_k) \rightarrow \kappa_{2+p}(C)$.

The Monotonicity axiom follows from Proposition \ref{margmono}.

The Permutation Invariance, Duality, and RSP axioms all follow from the analogs to Propositions \ref{fAmoc} \eqref{fAmoc.1}, \eqref{fAmoc.2}, and \eqref{fAmoc.3} for $\fB_2(C_T)$.

We wish finally to establish the TP axiom.  We recall that $r_{1+p} = 2p/(2+p)$ and notice that by \textbf{S5} we have $\kappa_2(\sigma_1^*(C_R)) = -\kappa_2(C_R)$ whenever $C_R$ is an extended 2-copula with $1 \notin R$.  Combining this last observation with Propositions \ref{proposition2} \eqref{prop2-7} and \eqref{prop2-8}, we see that 
\[
  \kappa_2(\sigma_1^*(C)_R) = 
    \begin{cases}
      &\kappa_2(C_R) \text{ if } 1 \in R \\
      -&\kappa_2(C_R) \text{ if } 1 \notin R.
    \end{cases}
\]
We then compute for a $(2+p)$-copula $C$,
\begin{align*}
  \kappa_{2+p}(C) &+ \kappa_{2+p}(\sigma_1^*(C)) = \frac{1}{\binom{2+p}{2}} \, \left[
    \underset{R \in \overline{2+p}(p)}{\sum} \kappa_2(C_R) + 
    \underset{R \in \overline{2+p}(p)}{\sum} \kappa_2(\sigma_1^*(C)_R) \right] \\
  =& \frac{1}{\binom{2+p}{2}} \left[ \underset{R \in \overline{2+p}(p)}{\sum} \kappa_2(C_R) + 
    \underset{\underset{1 \in R}{R \in \overline{2+p}(p)}}{\sum} \kappa_2(C_R) - 
    \underset{\underset{1 \notin R}{R \in \overline{2+p}(p)}}{\sum} \kappa_2(C_R) \right] \\
  =& \frac{2}{\binom{2+p}{2}} \, 
    \underset{\underset{1 \in R}{R \in \overline{2+p}(p)}}{\sum} \kappa_2(C_R) \\
  =& r_{1+p} \, \frac{1}{\binom{1+p}{2}} 
    \underset{T \in \overline{1+p}(p-1)}{\sum} \kappa_2(C_{\{1\}+T}) \\
  =& r_{1+p} \kappa_{1+p}(C_1).
\end{align*}
Thus the result is established.
\end{proof}

\begin{note}
Such extensions are not unique.  For example, Nelsen's extensions of Spearman's rho and Kendall's tau in \cite{Nelsen02} differ from the construction used here.
\end{note}

\subsection{A reflection-reduction formula}

We show that $\kappa_n((\sigma_{i_1}\cdots\sigma_{i_s})^*(C))$ can be written in terms
of the measures of concordance of the marginals of $C$.

\begin{theorem} \label{refreduce}
   Suppose that $C$ is an $n$-copula, that $S$ and $T$ are disjoint subsets of $\nbar$, and
  that $s$=card\,($S$) and $t$=card\,($T$).  Then 
  \begin{equation} \label{refreduce.1}
    r_{n-t} \kappa_{n-t}(\sigma_S^*(C_T)) = \sum_{j=n-t-s}^{n-t}(-1)^{n-t-s+j}
    r_{n-t}r_{n-t-1}{\cdots}r_j \sum_{P{\in}S(n-t-j)} \kappa_j(C_{T+P}),
  \end{equation}
  or, equivalently, 
  \begin{equation} \label{refreduce.2}
    \fA_{n-t}(\sigma_S^*(C_T)) = \sum_{j=n-t-s}^{n-t}(-1)^{n-t-s+j} \; \fA_j^S(C_T).
  \end{equation}  
\end{theorem}

We write out some examples of this formula using the
$(\sigma_{i_1}\cdots\sigma_{i_s})^*(C)$ and $C_{j_1{\cdots}j_k}$
notation.  This sort of notation may well seem more natural and easier
to understand when carrying out low-dimensional calculations.  On the
other hand, the $\sigma_S^*(C)$ and $\fA_{j}^S(C_T)$ notations can be more useful for such
tasks as expressing the general formula (\ref{refreduce.1}) concisely 
and correctly.

\begin{example}
  Suppose that $C$ is an $n$-copula and
  $S=\{i_1,\cdots,i_s\}{\subseteq}\nbar$. 

  The $s=1$ case:
  \[
    r_n \kappa_n(\sigma_{i_1}^*(C)) = 
      r_n r_{n-1} \kappa_{n-1}(C_{i_1}) - r_n \kappa_n(C).
  \]
  This is simply the Transposition Property trivially rearranged.

  The $s=2$ case:
  \begin{align*}
    r_n \kappa_n((\sigma_{i_1} \sigma_{i_2})^*(C))& = 
      r_n r_{n-1} r_{n-2} \kappa_{n-2}(C_{i_1 i_2}) \\
    -& r_n r_{n-1} \left[ \kappa_{n-1}(C_{i_1}) + \kappa_{n-1}(C_{i_2})
      \right] + r_n \kappa_n(C).
  \end{align*}

  The $s=3$ case:
  \begin{align*}
    r_n \kappa_n((\sigma_{i_1} \sigma_{i_2} \sigma_{i_3})^*(C)) &= 
      r_n r_{n-1} r_{n-2} r_{n-3} \kappa_{n-3}(C_{i_1 i_2 i_3}) \\
   -& r_n r_{n-1} r_{n-2} \left[ \kappa_{n-2}(C_{i_1 i_2}) +
      \kappa_{n-2}(C_{i_1 i_3}) + \kappa_{n-2}(C_{i_2 i_3})\right] \\
   +& r_n r_{n-1} \left[ \kappa_{n-1}(C_{i_1}) + \kappa_{n-1}(C_{i_2})
      + \kappa_{n-1}(C_{i_3}) \right] \\
   -& r_n \kappa_n(C).
  \end{align*}
\end{example}

\begin{proof}[Proof of Theorem \ref{refreduce}]  
  We proceed by induction with respect to $s$.  

  The result is trivially true for $s=0$.  Suppose it has been established for $s$ and we wish to extend it to $s+1$.  Let $S$ be a subset of $\nbar$ such that $s = \text{card\,}(S)$ and let $i$ be an element of $\nbar$ such that $i \notin S$.  We will establish the formula for $\fA_{n-t}^{\{i\}+S}(C_T)$.

  We first notice that we can combine Proposition \ref{basicfA}\,\eqref{basicfA.4} and Proposition \ref{fAmoc}\,\eqref{fAmoc.4} to obtain 
\begin{equation} \label{refset}
  \fA_j^S(\sigma_i^*(C_T)) = \fA_{j-1}^{\{i\}+S}(C_T) - \fA_{j-1}^S(C_T) - \fA_j^S(C_T).
\end{equation}
We now make use of our induction hypothesis and (\ref{refset}):  
\begin{align*}
  \fA_{n-t}(\sigma_{\{i\}+S}^*(C_T)) &= \fA_{n-t}(\sigma_S^*(\sigma_i^*(C_T))) \\
    =& \sum_{j=n-t-s}^{n-t}(-1)^{n-t-s-j} \; \fA_j^S(\sigma_i^*(C_T)) \\
    =& \sum_{j=n-t-s}^{n-t}(-1)^{n-t-s-j} [\fA_{j-1}^{\{i\}+S}(C_T) - \fA_{j-1}^S(C_T) - \fA_j^S(C_T)].
\end{align*}
This last sum, when written out, is found to telescope to  
\[
  \left(\sum_{n-t-s-1}^{n-t-1}(-1)^{n-t-s-1+j} \fA_{j}^{\{i\}+S}(C_T)\right) - \fA_{n-t-s-1}^S(C_T) + (-1)^{s+1}\: \fA_{n-t}^S(C_T).
\]
Now $\fA_{n-t-s-1}^S(C_T) = 0$ by Proposition \ref{basicfA}, part \eqref{basicfA.2}, and $ \fA_{n-t}^S(C_T) = r_{n-t} \kappa_{n-t}(C_T) =  \fA_{n-t}^{\{i\}+S}(C_T)$ by Proposition \ref{basicfA}, part \eqref{basicfA.1}; thus the theorem is established.
\end{proof}

\subsection{A complementarity principle}

\begin{theorem} \label{complement}
Let $C$ be an $n$-copula and $R$, $S$, and $T$ disjoint subsets of $\nbar$ such that 
\[
  \nbar = R+S+T, \quad r = \text{card\,}(R), \quad s = \text{card\,}(S), 
	\text{ and } t = \text{card\,}(T).
\]
Then for every measure of concordance $\kappa$ we have 
\begin{equation} \label{complement-eq}
  \sum_{j=s}^{n-t} (-1)^{r+j} \, \fA_j^R(C_T) = \sum_{j=r}^{n-t} (-1)^{s+j} \, \fA_j^S(C_T).
\end{equation}
\end{theorem}

\begin{proof}
We know by Proposition \ref{mocmarg}\eqref{mocmarg.2} that 
\[
  r_{n-t} \kappa_{n-t}(\sigma_R^*(C_T)) = r_{n-t} \kappa_{n-t}((\sigma_R \sigma_{R+S})^*(C_T)) 
    = r_{n-t} \kappa_{n-t}(\sigma_S^*(C_T)).
\]
By Theorem \ref{refreduce} we then have 
\[
  \sum_{j=n-t-r}^{n-t} (-1)^{n-t-r+j} \, \fA_j^R(C_T) = \sum_{j=n-t-s}^{n-t} (-1)^{n-t-s+j} \, \fA_j^S(C_T),
\]
which is the desired result.
\end{proof}

\begin{example}
  Suppose we take $C$ to be a 4-copula, $R = \{1,2\}$, $S = \{3,4\}$, and $T = \emptyset$.  Then (\ref{complement-eq}) is 
\[
  \sum_{j=2}^{4} (-1)^{j} \, \fA_j^R(C) = \sum_{j=2}^{4} (-1)^{j} \, \fA_j^S(C),
\]
which, since $\fA_4^R(C) = \fA_4^S(C)$ by Proposition \ref{basicfA}\eqref{basicfA.1}, reduces to 
\[
  \sum_{j=2}^{3} (-1)^{j} \, \fA_j^R(C) = \sum_{j=2}^{3} (-1)^{j} \, \fA_j^S(C).
\]
Recalling that $r_2 = 2/3$ and making the (very reasonable) assumption that $r_3, r_4 \ne 0$, this further reduces to 
\[
  \kappa_3(C_1) + \kappa_3(C_2) - \frac{2}{3} \, \kappa_2(C_{12}) = 
    \kappa_3(C_3) + \kappa_3(C_4) - \frac{2}{3} \, \kappa_2(C_{34}).
\]
\end{example}

\begin{note}
The existence of relations such as the one in this example suggest that it may be possible to use Theorem \ref{complement} to investigate the compatibility of marginals for a copula, a question which is certainly of interest.  We do not pursue that line of thought here.
\end{note}

\subsection{\'{U}beda identities}

It was brought to our attention by M. \'{U}beda Flores that if $C$ is a 3-copula, then 
\[
  \kappa_3(C) = \frac{1}{3} \, (\kappa_2(C_1) + \kappa_2(C_2) + \kappa_2(C_3)).
\]
This turns out to be the first in an infinite list of identities in which the measure of concordance of an odd-dimensional copula can always be expressed in terms of the measures of concordance of its even-dimensional marginals.  In general these identities involve the constants $\{r_n\}$; we have a $1/3$ in the identity for the 3-copula because the fact that $r_2 = 2/3$.

\begin{theorem} \label{ubeda}
  Let $\{\gamma_k\}$ be the sequence of constants defined by 
\begin{align*}
  &\gamma_1 = \frac{1}{2}, \\
  &\gamma_{p+1} = \frac{1}{2} \, \left( 1 - \sum_{k=1}^p \gamma_k \binom{2p+1}{2k-1} \right), 
    \text{ for } p \geq 1.
\end{align*}
If $C$ is an $n$-copula, $R \subseteq \nbar$ with $r = \text{card}(R)$, and $n-r \geq 2p+1$, then for $p \geq 1$ we have 
\begin{equation} \label{ubident}
  \fA_{2p+1}(C_R) = \sum_{j=1}^p \gamma_{p-j+1} \binom{n-r-2j}{2p-2j+1} \fA_{2j}(C_R).
\end{equation}
\end{theorem}

\begin{proof}
  We may, without loss of generality, assume $R = \emptyset$ and prove (\ref{ubident}) with $r=0$.  The transition to the case involving $C_R$ is then made by noting that (\ref{ubident}) holds for the proper copula of $C_R$.

\bigskip

\textbf{Step 1.}  Suppose $C$ is a $(2p+1)$-copula.  Then by Proposition \ref{fAmoc}\eqref{fAmoc.2} and Theorem \ref{refreduce}, we have 
\[
  \fA_{2p+1}(C) = \fA_{2p+1}(\sigma_{\overline{2p+1}}^*(C)) = 
  \fA_2(C) - \fA_3(C) + \cdots - \fA_{2p+1}(C),
\]
from which it follows that 
\begin{equation} \label{weakub}
   \fA_{2p+1}(C) = \frac{1}{2} \, \left( \sum_{j=1}^p \fA_{2j}(C) - 
    \sum_{j=2}^p \fA_{2j-1}(C) \right).
\end{equation}
Equation (\ref{weakub}) will also hold if $C$ is not a $(2p+1)$-copula but rather an extended $(2p+1)$-copula.  That is, it will hold if $C = E_R$ where $E$ is an $n$-copula with $n > 2p+1$ and $R$ is a subset of $\nbar$ with $\text{card}(R) = n-2p-1$.

\bigskip

\textbf{Step 2.}
We verify (\ref{ubident}) for the case $p=1$.  

Let $C$ be an $n$-copula with $n \geq 3$ and choose $R \in \overline{n}(n-3)$.  Let $T = \overline{n}-R$.  By Proposition \ref{fAmoc}\eqref{fAmoc.2}, we have $\fA_3(\sigma_T^*(C_R)) = \fA_3(C_R)$.  It then follows from Theorem \ref{refreduce} and $\kappa_0 = \kappa_1 = 0$ that
\[
	\fA_3(C_R) = \sum_{j=0}^3 (-1)^j \fA_j(C_R) = \fA_2(C_R) - \fA_3(C_R).
\]
Thus $\fA_3(C_R) = \gamma_1 \fA_2(C_R)$.

We next consider $\fA_3(C)$.  
\begin{align*}
	\fA_3(C) &= r_n \cdots r_3 \, \sum_{R \in \overline{n}(n-3)} \kappa_3(C_R) \\
	=& \, r_n \cdots r_4 \, \sum_{R \in \overline{n}(n-3)} \fA_3(C_R) 
		 \quad \text{($r_n \cdots r_4 = 1$ if $n=3$)} \\
	=& \, \gamma_1 r_n \cdots r_4 \, \sum_{R \in \overline{n}(n-3)} \fA_2(C_R) 
		 \quad \text{(result above)} \\
	=& \,  \gamma_1 r_n \cdots r_4 r_3 r_2 \, \sum_{R \in \overline{n}(n-3)} \sum_{S \in (\overline{n}-R)(1)} \kappa_2(C_{R+S})  \quad \text{(definition of $\fA$)} \\
	=& \,  \gamma_1 r_n \cdots r_2 \, \binom{n-2}{1} \, \sum_{T \in \overline{n}(n-2)} \kappa_2(C_{R+S})  
		\quad \text{(Lemma \ref{countinglemma}\eqref{counting.2})} \\
	=& \, \gamma_1 \, \binom{n-2}{1} \, \fA_2(C).
\end{align*}

We conclude Step 2 by noting that if $C = E_R$ where $E$ is an $n$-copula and $r = \text{card}(R)$, then
\[
	\fA_3(E_R)  = \gamma_1 \, \binom{n-r-2}{1} \, \fA_2(E_R).
\]

\bigskip

\textbf{Step 3.}
Suppose that for a given $p \geq 2$ we have verified equation (\ref{ubident}) for $\fA_{2q+1}(E_R)$ for all $m$-copulas $E$ and $R \subseteq \mbar$ for which $r = \text{card}(R)$, $2q+1 \leq m-r$, and $q \leq p-1$.   We then take $C$ to be an $n$-copula with $2p+1 \leq n$ and calculate thus:
\begin{align*}
  \fA_{2p+1}(C) &= r_n \cdots r_{2p+2} \underset{R \in \nbar(n-2p-1)}{\sum} \fA_{2p+1}(C_R) 
    \quad \text{(Proposition \ref{basicfA}\eqref{basicfA.1})} \\
  =& \, r_n \cdots r_{2p+2} \underset{R \in \nbar(n-2p-1)}{\sum} \frac{1}{2} 
    \left[ \sum_{j=1}^p \fA_{2j}(C_R) - \sum_{k=2}^p \fA_{2k-1}(C_R) \right] 
    \quad \text{(equation (\ref{weakub}))} \\
  =& \, \frac{1}{2} \, r_n \cdots r_{2p+2} \underset{R \in \nbar(n-2p-1)}{\sum}
    \left[ \sum_{j=1}^p \fA_{2j}(C_R) - \sum_{k=2}^p \sum_{j=1}^{k-1} \gamma_{k-j} 
    \binom{2p-2j+1}{2k-2j-1} \fA_{2j}(C_R) \right] \\
  &\quad \quad \text{(induction hypothesis)} \\
  =& \, \frac{1}{2} \, r_n \cdots r_{2p+2} \underset{R \in \nbar(n-2p-1)}{\sum}
    \left[ \sum_{j=1}^p \fA_{2j}(C_R) - \sum_{j=1}^{p-1} \sum_{k=j+1}^{p} \gamma_{k-j} 
    \binom{2p-2j+1}{2k-2j-1} \fA_{2j}(C_R) \right] \\
  =& \, \frac{1}{2} \, r_n \cdots r_{2p+2} \underset{R \in \nbar(n-2p-1)}{\sum}
    \left[ \fA_{2p}(C_R) + \sum_{j=1}^{p-1} \left( 1 - 
    \sum_{i=1}^{p-j} \gamma_i \binom{2(p-j)+1}{2i-1} \right) \fA_{2j}(C_R) \right] \\
  =& \, r_n \cdots r_{2p+2} \underset{R \in \nbar(n-2p-1)}{\sum} 
    \quad \sum_{j=1}^p \gamma_{p-j+1} \fA_{2j}(C_R) \\
  =& \sum_{j=1}^p \gamma_{p-j+1} \binom{n-2j}{n-2p-1} \fA_{2j}(C) 
    \quad \text{(Proposition \ref{basicfA}\eqref{basicfA.6})}.
\end{align*}
Thus the result is established.
\end{proof}

\begin{example}
  One readily calculates that 
\[
  \gamma_1 = \frac{1}{2}, \quad \gamma_2 = - \frac{1}{4}, \quad 
    \gamma_3 = \frac{1}{2}, \quad \gamma_4 = - \frac{17}{8}.
\]
Then from equation (\ref{ubident}) one sees the following:

If $C$ is a 3-copula, then 
\[
  r_3 \kappa_3(C) = \frac{1}{2}\, \fA_2(C).
\]

If $C$ is a 5-copula, then 
\[
  r_5 \kappa_5(C) = - \frac{1}{4}\, \fA_2(C) + \frac{1}{2}\, \fA_4(C).
\]

If $C$ is a 7-copula, then 
\[
  r_7 \kappa_7(C) = \frac{1}{2}\, \fA_2(C) - \frac{1}{4}\, \fA_4(C) + \frac{1}{2}\, \fA_6(C).
\]

If $C$ is a 9-copula, then 
\[
  r_9 \kappa_9(C) = -\frac{17}{8}\, \fA_2(C) + \frac{1}{2}\, \fA_4(C) - \frac{1}{4}\, \fA_6(C) + \frac{1}{2}\, \fA_8(C).
\]

One may ``divide out'' $r_{2p+1}$ in these examples to obtain formulas of the form ``$\kappa_{2p+1}(C)=\cdots$.''  We can do this not because we know $r_{2p+1} \ne 0$ but because one can repeat all the underlying, relevant proofs with $\fA_j(C_R)$ replaced with expressions without the leading $r_n$ and of the form 
$r_{n-1}r_{n-2}\cdots r_j\underset{T \in (\nbar - R)(n-r-j)}{\sum}\kappa_j(C_{R+T})$.  The situation is analogous to that for $\fB_j^S(C_T)$.
\end{example}

\subsection{An asymptotic result}

Suppose that we compute $\kappa_s(C_s)$ for the copula $C_s$ of the random vector $(X_1,\cdots,X_s)$ and then consider a new, enlarged random vector $(X_1,\cdots,X_s,X_{s+1},\cdots,X_{s+p})$ where each new random variable $X_{s+k}$ is a monotone increasing function of, let us say, $X_1$.  If $C_{s+p}$ is the copula of the new random vector, then one is tempted to suspect that $\kappa_{s+p}(C_{s+p}) \rightarrow 1$ as $p \rightarrow \infty$.  However, this is often not the case.

To see this, consider the random vector $(\epsilon_1 X,\epsilon_2 X,\cdots,\epsilon_n X)$, where $X$ is a fixed, continuous random variable, $\epsilon_i = -1$ for exactly $s$ of the $i$'s, and $\epsilon_i = 1$ for all the other $i$'s.  The $n$-copula of this random vector is $\sigma_S^*(M)$ for some $S \subseteq \nbar$ with $s = \text{card}(C)$.  We know by Corollary \ref{Mmoc} that $\kappa_n(\sigma_S^*(M))$ depends only on $n$ and $s$ and not on the particular subset $S$ of $\nbar$ with which we are working.  We can investigate the question raised above by considering the behavior of $\kappa_n(\sigma_S^*(M))$ if $s$ is held fixed and $n \rightarrow \infty$.

\begin{theorem}
  If $s = \text{card}(S) > 0$ where $S \subseteq \nbar$, then 
\begin{equation} \label{Mformula}
  \kappa_n(\sigma_S^*(M)) = (-1)^s + \sum_{k=1}^s (-1)^{k+s}\binom{s}{k} \,  r_{n-1} \cdots r_{n-k}.
\end{equation}
\end{theorem}

\begin{proof}
  Notice that for $n-s \leq j \leq n$ we have 
\[
  \fA_j^S(M) = r_n \cdots r_j \underset{R \in S(n-j)}{\sum} \kappa_j(M_R) = 
    \binom{s}{n-j} \, r_n \cdots r_j.
\]
Then from Theorem \ref{refreduce} we obtain 
\begin{gather*}
  r_n \kappa_n(\sigma_S^*(M)) = \sum_{j=n-s}^n (-1)^{n-s+j} \binom{s}{n-j} \, r_n \cdots r_j \\
  = \sum_{k=0}^s (-1)^{k+s} \binom{s}{k} \, r_n \cdots r_{n-k}.
\end{gather*}
We can ``divide out'' $r_n$ as we did when dealing with the $\fB_j^S(C_T)$ notation.  This gives us equation (\ref{Mformula}).
\end{proof}

\bigskip

\begin{corollary}
  Suppose in the theorem above that $r_n \rightarrow r$.  If $s$ is held fixed while $n \rightarrow \infty$, we see that 
\[
  r_n \kappa_n(\sigma_S^*(M)) \rightarrow (-1)^s (1-r)^s.
\]
\end{corollary}

\bigskip

Now for the examples of measures of concordance in \cite{taylor04a} for which it is clear that $r_n \rightarrow r$---for instance, Nelsen's generalizations of Spearman's rho and Kendall's tau---we find that $r=1$.  Thus for those examples, $\kappa_n(\sigma_S^*(M)) \rightarrow 0$ provided $s>0$.  However, if we look at the construction used here to extend a bivariate measure of concordance to a multivariate one, we see that 
\[
  r_{1+p} = \frac{2p}{2+p} \rightarrow 2 \quad \text{as} \quad p \rightarrow \infty,
\]
so that $\kappa_n(\sigma_S^*(M)) \rightarrow 1$ as $n \rightarrow \infty$.

\subsection{Questions}

We briefly consider a few questions or possible lines of investigation that suggest themselves.  We take $\kappa = (\{\kappa_n\},\{r_n\})$ as a measure of concordance.

\begin{enumerate}
  \item  For all examples with which we are familiar, $r_n > 0$ for all $n \geq 2$.  Is this true in general?  We also find for all our examples that there is an $r$ such that $r_n \rightarrow r$.  Again, is this true in general?

  \item  Find the minimum value of $\kappa_n(C)$ where $C \in \text{Cop}(n)$.  Does this always occur at some $C = \sigma_S^*(M)$?  Does the minimum depend only on the $r_n$'s or does the choice of $\kappa$ matter?

  \item  Given an $n$-copula $A$, find an $(n+1)$-copula $B$ such that $B_1 = A$ and $\kappa_{n+1}(B)$ is a maximum (or a minimum).

  \item  Given an $n$-copula $C$, can we specify a minimal, finite set $\{\kappa_m(C_R)\}$ of measures of concordance of marginals of $C$ which would be sufficient to compute $\kappa_p(\sigma_S^*(C_T))$ for all $p$ and all $S,T \subseteq \nbar$?  Given such a set $\{\kappa_m(C_R)\}$, can we find an ``efficient'' algorithm to compute the maximum and the minimum values of $\kappa_p(\sigma_S^*(C_T))$?

  \item  The examples of measures of concordance with which we are familiar have the 
    property that if $A$ and $B$ are $n$-copulas and $0 \leq t \leq 1$, 
    then $\kappa_n((1-t)A+tB)$ is either a first or second degree polynomial 
    in $t$.  One can clearly extend this idea to talk about measures of 
    concordance of degree $m$.  Is it possible to characterize measures 
    of concordance of a fixed degree, to exhibit some sort of canonical 
    form for them?  (This has been done for bivariate measures of concordance of degree one in \cite{Edwards04}, \cite{Edmiktay03a}, and \cite{Edmiktay03b}.)

\end{enumerate}

\end{document}